\documentclass{article} 
\usepackage{iclr2027_conference,times}

\usepackage{amsmath,amsfonts,bm}

\def\eqref#1{equation~\ref{#1}}
\def\Eqref#1{Eq.~(\ref{#1})}

\def\1{\bm{1}}

\DeclareMathAlphabet{\mathsfit}{\encodingdefault}{\sfdefault}{m}{sl}
\SetMathAlphabet{\mathsfit}{bold}{\encodingdefault}{\sfdefault}{bx}{n}

\usepackage{amsthm} 
\newtheorem{lemma}{Lemma}[section]
\newtheorem{proposition}{Proposition}[section]

\usepackage{booktabs,tabularx,placeins,float}
\usepackage{graphicx,subcaption}
\usepackage{microtype}
\usepackage{wrapfig}
\usepackage{needspace}
\newcommand{\finishwrap}{%
    \par
    \ifnum\value{WF@wrappedlines}>1
        \vspace{\dimexpr\baselineskip*(\value{WF@wrappedlines}-1)\relax}%
    \fi
    \WFclear
}
\usepackage{xcolor}
\usepackage{hyperref}
\hypersetup{hidelinks}
\usepackage{url}
\usepackage{algorithm,algpseudocode}
\DeclareMathOperator{\sig}{sig}
\DeclareMathOperator{\clip}{clip}

\title{Reinforcement Learning to Accelerate Primal-Dual Hybrid Gradient for Linear Programming}

\author{Jinhwan Sul \\
Georgia Institute of Technology, USA \\
\texttt{jsul7@gatech.edu} \\
\And
Alex Oshin \\
Georgia Institute of Technology, USA \\
\texttt{alexoshin@gatech.edu} \\
\AND
Evangelos A. Theodorou \\
Georgia Institute of Technology, USA \\
\texttt{evangelos.theodorou@gatech.edu}
}

\iclrfinalcopy
\begin{document}

\maketitle
\fancyhead{}
\renewcommand{\headrulewidth}{0pt}

\begin{abstract}
Primal--dual hybrid gradient (PDHG) methods solve large-scale linear programs (LPs) using GPU-friendly matrix--vector products and projections, but their practical performance depends on coordinating algorithm parameters, acceleration, and restarts. We introduce GALLOP, which uses reinforcement learning to jointly learn continuous algorithm parameters and discrete restart decisions without differentiating through the solver. Its generalized accelerated PDHG update combines separate primal and dual extrapolation, history corrections, and restart anchoring with independently adjustable coefficients. We train a dimension-agnostic feedback policy using a groupwise proximal policy optimization objective that clips likelihood ratios separately for different control groups and excludes inactive acceleration controls on restart transitions. We evaluate GALLOP on six LP families and a public item-placement benchmark. On the main evaluation settings across the six families, GALLOP reduces iteration counts by factors of $1.9$--$5.6$ and achieves up to a $16.0\times$ speedup in algorithm wall-clock time over MPAX. With one policy trained per family, the learned policies generalize without retraining to within-family LPs $3\times$--$400\times$ larger than the largest training instances, including Transport LPs with $10.24$ million variables.
\end{abstract}

\section{Introduction}
\label{sec:intro}

Linear programming (LP) is a fundamental problem class in constrained optimization~\citep{dantzig1963linear}.
The simplex and interior-point methods are widely used to solve LPs~\citep{dantzig1963linear,wright1997primal}.
Recently, due to the scalability limitations of these traditional methods~\citep{applegate2026pdlp},
primal--dual hybrid gradient (PDHG) method~\citep{chambolle2011first} based LP algorithms have been explored for large-scale problems~\citep{applegate2021practical,lu2023cupdlp,lu2025cupdlpx}.
PDHG iteration consists of matrix--vector products and projections which are inexpensive and GPU-friendly.
However, PDHG can require a large number of iterations to attain high-accuracy solutions~\citep{applegate2023faster}.

PDHG-based LP solvers~\citep{lu2025cupdlpx,zhou2025anderson} improve their practical convergence through heuristic adjustment of algorithmic parameters, such as the step size and primal weight, and acceleration techniques, such as restarted Halpern acceleration~\citep{halpern1967fixed,lu2024halpern} and Anderson acceleration~\citep{anderson1965iterative}.
Many of these heuristics and acceleration schemes are motivated by theoretical convergence analyses~\citep{lu2024halpern,applegate2023faster,zhou2025anderson}.
However, existing global convergence guarantees do not cover the full combination of heuristic adaptations used in practical solvers, whose effectiveness has been demonstrated empirically~\citep{lu2023cupdlpc,lu2025cupdlpx}.
Beyond convergence guarantees, a separate design challenge is to determine how these adaptations should be selected and coordinated to reduce solution time for a given family of LPs.
This leaves an opportunity to further accelerate PDHG-based LP solvers by learning algorithm parameters and heuristic rules from LP data.

Learning to optimize (L2O) uses problem data to accelerate optimization~\citep{Gregor2010Learning}.
We focus on L2O methods that use a learned policy throughout the solve until convergence, without handing the remaining iterations to an external solver.
\citet{venkataraman2021neural} and \citet{Martin2026Learning} use algorithm unrolling for fixed-point optimization.
However, their neural network designs are tied to the problem dimension and therefore do not support size generalization.
PDHG-Net~\citep{Li2024PDHGUnrolled} unrolls the basic PDHG iteration without incorporating discrete restart decisions or the additional acceleration schemes used by practical PDHG-based solvers.
It does not run the learned updates until convergence, but instead uses the network output to warm-start PDLP.

RLQP~\citep{Ichnowski2021Accelerating} uses reinforcement learning (RL) to learn the penalty parameter in OSQP~\citep{stellato2020osqp}.
While effective, RLQP applies model-free RL to tune continuous parameters of an iterative algorithm whose updates can be differentiated. 
In this setting, deep unfolding enables gradients of the optimization objective to be propagated directly through the solver, avoiding the additional value-function estimation and exploration required by RL. 
For example, DeepQP~\citep{saravanos2025deep} unfolds OSQP and learns its algorithm parameters end-to-end, providing a more direct gradient-based alternative.

\begingroup
\setlength{\intextsep}{2pt}
\begin{wrapfigure}{R}{0.65\linewidth}
    \begin{minipage}[t]{\linewidth}
        \centering
        \includegraphics[width=\linewidth]{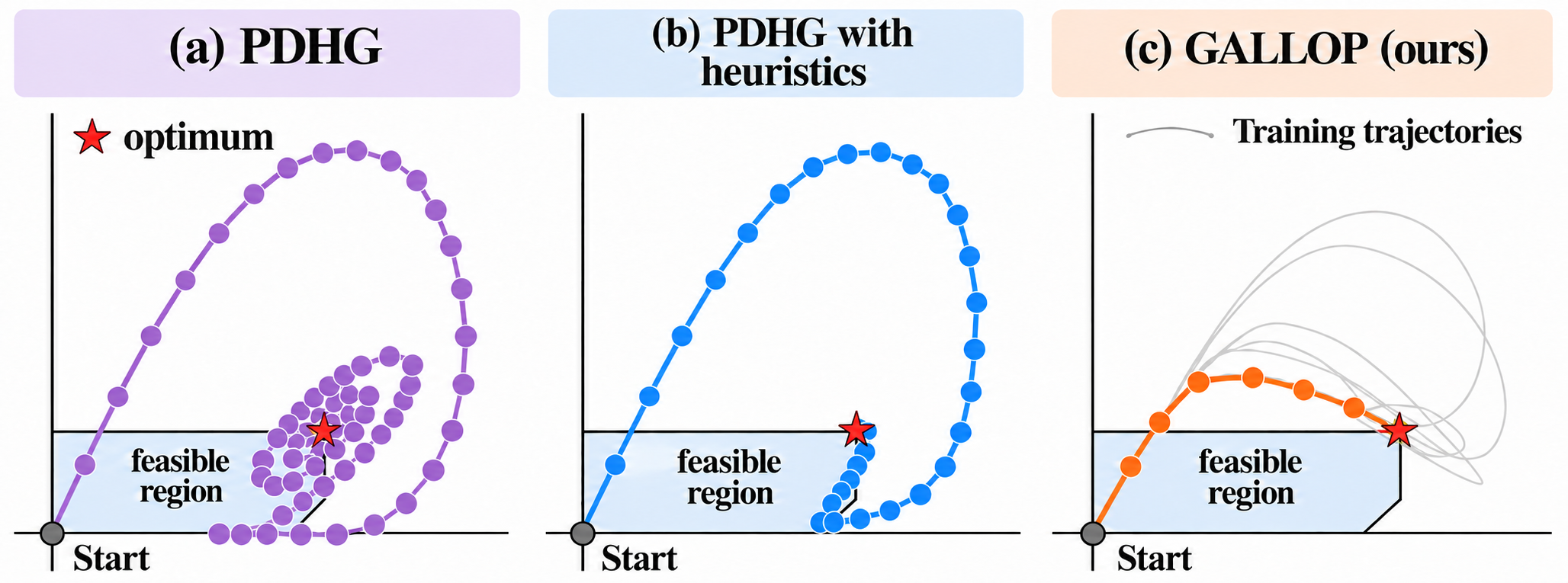}
        \captionsetup{type=figure,font=small,justification=raggedright,singlelinecheck=false,hypcap=false,skip=0pt}
        \caption{Conceptual comparison of (a) vanilla PDHG, (b) PDHG with heuristic acceleration and restarts, and (c) GALLOP with learned solver parameters and restarts.}
        \label{fig:concept}
    \end{minipage}
\end{wrapfigure}
\noindent
To solve large LPs faster, we learn a policy that jointly controls continuous algorithm parameters and discrete restart decisions in a PDHG-based LP algorithm.
We use RL to learn these controls jointly because ordinary backpropagation through the solver does not directly provide gradients for hard restart decisions.
The policy uses a fixed number of solver-state features and outputs a fixed number of controls, independent of the problem size.
This dimension-agnostic design allows us to train on smaller LPs and apply the learned policy to larger LPs without retraining.
\par\endgroup

We augment PDHG with a generalized acceleration step and a restart mechanism, and refer to the resulting update scheme as \textbf{generalized accelerated PDHG (GA-PDHG)}.
This step combines separate primal and dual extrapolation along the current PDHG update, history corrections from recent iterate differences, and a correction toward the restart anchor.
We learn the coefficients of these terms together with the step size, primal weight, and restart decision.
These controls provide a broad action space for RL to explore combinations of acceleration and restarting.
We call our proposed algorithm \textbf{GALLOP}, short for \textbf{G}eneralized \textbf{A}cceleration and Reinforcement \textbf{L}earning for \textbf{L}inear \textbf{O}ptimization with \textbf{P}DHG.

In experiments on six LP families, we demonstrate GALLOP is $1.9\times$--$5.6\times$ faster in terms of number of iterations relative to MPAX's restarted reflected Halpern solver~\citep{lu2024mpax}.
Algorithm wall-clock speedups reach $16.0\times$ on Random LPs with $300{,}000$ variables, $10.1\times$ on Auction LPs with $358{,}400$ variables, and $3.7\times$ on Transport LPs with $10.24$ million variables.
Our trained policies generalize to problems $3\times$--$400\times$ larger than the largest training size in a zero shot setting, using the family-specific size measures.
In particular, the Transport policy trained on LPs with at most $25{,}600$ variables, scales to LPs with $10.24$ million variables, a $400\times$ size generalization without retraining.
Ablations on Transport show that all five learned control groups contribute to reducing iteration counts, including the generalized-acceleration coefficients.
Our main contributions are as follows.
\begin{itemize}
    \setlength{\itemsep}{3pt}
    \setlength{\parsep}{0pt}
    \item We propose \textbf{GA-PDHG}, a generalized accelerated PDHG update scheme that combines independently controlled primal and dual extrapolation, history corrections, and a restart-anchor correction. This parameterization encompasses established acceleration schemes and their combinations, providing a broad search space for RL to discover acceleration strategies.
    \item We introduce \textbf{GALLOP}, an RL-based L2O framework for large-scale LPs. To accommodate the conditional discrete-continuous action space of GA-PDHG, we develop a groupwise PPO objective that clips likelihood ratios separately for each control group and masks inactive acceleration controls on restart transitions.
    \item We demonstrate acceleration and size generalization across six LP families, with iteration-count reductions by factors of $1.9$--$5.6$ and up to a $16.0\times$ speedup in algorithm wall-clock time over MPAX. The learned policies generalize to LPs $3\times$--$400\times$ larger than the largest training instances under family-specific size measures, including Transport LPs with $10.24$ million variables, without retraining.
\end{itemize}

\section{Primal--Dual Hybrid Gradient for Large-Scale Linear Programming}
\label{sec:pdhg}

PDLP~\citep{applegate2021practical} is a first-order LP solver based on the primal--dual hybrid gradient (PDHG) method~\citep{chambolle2011first}.
Its updates use matrix--vector products and projections, avoiding matrix factorization.
Basic PDHG can require many iterations to reach high accuracy.
When PDHG is applied to LP, acceleration techniques and heuristics are introduced to reduce the number of iterations.
This section presents the LP formulation and PDHG updates and reviews these enhancements in different PDLP variants.

\subsection{Problem Formulation: Linear Programming}

We consider the LP
\begin{equation}
\label{eq:LP}
\min_{x\in X} \quad c^\top x \quad \text{subject to} \quad Px = q, \quad Gx \ge h, \quad X = \{x : l \leq x \leq u\},
\end{equation}
where \(x \in \mathbb{R}^n\) is the primal variable and \(c \in \mathbb{R}^n\) is the objective coefficient vector. The LP has \(m_e\) equality constraints specified by \(P \in \mathbb{R}^{m_e \times n}\) and \(q \in \mathbb{R}^{m_e}\), and \(m_i\) inequality constraints specified by \(G \in \mathbb{R}^{m_i \times n}\) and \(h \in \mathbb{R}^{m_i}\).
The primal variable $x$ is bounded with $X$ where
$l\in\left(\mathbb R\cup\{-\infty\}\right)^n$, $u\in\left(\mathbb R\cup\{+\infty\}\right)^n$, and $l\leq u$ componentwise.
Define \(A = [P^\top \; G^\top]^\top \in \mathbb{R}^{m \times n}\) and \(b = [q^\top \; h^\top]^\top \in \mathbb{R}^m\), where \(m = m_e + m_i\).
The corresponding dual vector is \(y = [\left(y^{\mathrm{eq}}\right)^\top \; \left(y^{\mathrm{in}}\right)^\top]^\top \in \mathbb{R}^m\), with equality multipliers \(y^{\mathrm{eq}} \in \mathbb{R}^{m_e}\) and inequality multipliers \(y^{\mathrm{in}} \in \mathbb{R}_+^{m_i}\).
Linear programs (LPs) are commonly solved with the simplex~\citep{dantzig1963linear} or interior-point methods~\citep{wright1997primal}.
Standard implementations of these methods include matrix factorizations, whose computational and memory costs can become prohibitive at large scales.

\subsection{Primal--Dual Hybrid Gradient for Linear Programming}
\label{sec:pdlp-acceleration}

PDHG uses proximal gradients of the Lagrangian to update the primal and dual variables without matrix factorization. To derive these first-order updates, we express \Eqref{eq:LP} as the equivalent convex--concave min-max problem
\begin{equation}
\label{eq:saddle}
    \min_{x\in X}\ \max_{y\in Y}\quad L(x,y) := c^\top x + \left(y^{\mathrm{eq}}\right)^\top \left(q - Px\right) + \left(y^{\mathrm{in}}\right)^\top \left(h - Gx\right) = c^\top x - y^\top Ax + b^\top y,
\end{equation}
where $Y := \mathbb{R}^{m_e}\times\mathbb{R}_+^{m_i}$ leaves the equality multipliers unrestricted and the inequality multipliers nonnegative. For fixed $y\in Y$, minimizing the Lagrangian over $X$ gives the dual objective
\begin{equation}
\label{eq:lp-dual-objective}
    \inf_{x\in X}L\left(x,y\right)
    =b^\top y+\inf_{x\in X}\left(c-A^\top y\right)^\top x.
\end{equation}
For the dual objective to be finite, the reduced cost $c-A^\top y$ must satisfy the sign conditions imposed by $X$.
We denote the set of such reduced costs by $\Lambda$.

Throughout the paper, $z:=(x^\top,y^\top)^\top\in\mathbb R^{n+m}$ denotes a stacked primal--dual vector. Subscripts and accents on $z$ apply to both blocks $x$ and $y$. At iteration $t$, $\hat z_t$ is the PDHG input and $z_t\in X\times Y$ is its output.
At iteration $t$, projected descent and ascent with primal extrapolation give
\begin{subequations}
\label{eq:pdhg_iteration}
\begin{align}
    x_t &= \operatorname{proj}_X \left[ \hat x_t - \frac{\eta}{\omega}\left( c- A^\top\hat y_t \right) \right],  \\
    y_t &= \operatorname{proj}_Y\left[\hat y_t + \omega \eta \left( b - A\left(2x_t - \hat x_t \right) \right) \right].
\end{align}
\end{subequations}
Here $\operatorname{proj}$ denotes Euclidean projection, $\eta > 0$ is the step size, and $\omega > 0$ is the primal weight, giving primal and dual step sizes $\eta/\omega$ and $\eta\omega$, respectively.
Writing $\mathcal H$ for this PDHG map, we have $z_t=\mathcal H\left(\hat z_t;\omega,\eta\right)$.
With iteration-dependent parameters $\eta_t,\omega_t>0$, $z_{t+1}=\mathcal H\left(\hat z_{t+1};\omega_t,\eta_t\right)$.
Thus the current pair $\left(\hat z_t,z_t\right)$ was computed with $\left(\omega_{t-1},\eta_{t-1}\right)$, with $\left(\omega_0,\eta_0\right)$ reserved for initialization.
For any $\omega,\eta>0$, $z$ is a fixed point of the PDHG map, i.e. $\mathcal H\left(z;\omega,\eta\right)=z$, if and only if $z$ satisfies the KKT conditions of \Eqref{eq:LP}.


\paragraph{PDHG enhancements for LP.}
PDHG-based LP solvers choose a constant step size or adaptively choose it through trial-update acceptance tests~\citep{chambolle2011first,applegate2021practical}.
The primal weight is fixed or adjusted at restarts or fixed intervals to balance primal--dual displacements, typically through exponential moving averaging (EMA) in log space or proportional--integral--derivative (PID) control~\citep{lu2025cupdlpx}.
Restarting~\citep{hinder2020generic} resets the averaging or acceleration state from a selected iterate after sufficient progress, a loss of progress following an initial decrease, or an excessively long epoch. The progress is measured by the normalized duality gap, KKT error, or fixed-point residual~\citep{applegate2023faster,lu2024halpern}.
Acceleration uses step-size-weighted averaging~\citep{applegate2021practical}, Halpern anchoring with a diminishing anchor weight and optional reflection, or Anderson history corrections computed by least-squares residual minimization~\citep{anderson1965iterative,zhou2025anderson}.
With suitable parameter choices and restart rules, restarted averaging and Halpern acceleration exploit LP sharpness to achieve accelerated linear convergence~\citep{applegate2023faster,lu2024halpern}.
See Appendix~\ref{app:pdlp-enhancements} for the details on these enhancements across representative solvers.

\section{Reinforcement Learning for PDHG with Acceleration}
\label{sec:method}

The enhancement rules reviewed in PDHG-based LP solvers are hand-designed feedback controls based on solver state.
Their criteria address individual decisions, such as balancing primal--dual displacements, testing restart thresholds, or minimizing a local residual model.
These heuristics motivate exploring whether a more expressive policy that coordinates these decisions using a broader set of solver-state features can further accelerate convergence.
We therefore learn such a policy from solver trajectories using a common reward that induces faster convergence.

\subsection{PDHG with Generalized Acceleration}
\label{sec:generalized-acceleration}

To enable a learned policy to explore both established acceleration rules and their combinations, we introduce a family of accelerated PDHG updates.
Specifically, we combine extrapolation, history corrections, and anchoring with independently adjustable coefficients to formulate \textbf{generalized accelerated PDHG (GA-PDHG)}.
For a fixed, finite memory length $k\geq0$, define the output-history matrix $\mathbf K_t:=\begin{bmatrix}z_t-z_{t-1}&\cdots&z_{t-k+1}-z_{t-k}\end{bmatrix}\in\mathbb R^{\left(n+m\right)\times k}$.
Unavailable differences and those discarded at a restart are replaced by zero.
The vector $z_\star$ is the fixed anchor within the current epoch, and $s$ counts the non-restart iterations within the current epoch.
The coefficients $\alpha_t^x,\alpha_t^y\geq-1$ control primal and dual relaxation or extrapolation, $\boldsymbol\beta_t\in\mathbb R^k$ controls the history correction, and $\zeta_t\in\mathbb R$ controls the anchor correction.
Let $I_d$ denote the $d$-dimensional identity matrix and $e_t\in\{0,1\}$ the restart decision.
When $e_t=0$ (no restart), the GA-PDHG update is
\begin{subequations}
\begin{align}
    \hat z_{t+1}
    &=z_t+
    \begin{pmatrix}
        \alpha_t^x I_n&0\\
        0&\alpha_t^y I_m
    \end{pmatrix}\left(z_t-\hat z_t\right)
    +\mathbf K_t\boldsymbol\beta_t +\zeta_t
    \left(z_\star-z_t\right),
    \label{eq:gen_acc_iter}\\
    z_{t+1}&=\mathcal H\left(\hat z_{t+1};\omega_t,\eta_t\right).
    \label{eq:base_PDHG}
\end{align}
\end{subequations}
When $e_t=1$ (restart), the restart branch sets $\hat z_{t+1}=z_t$.
Note that GA-PDHG always ends with a PDHG mapping and takes its output as the iterate, allowing the acceleration coefficients to explore combinations beyond convex combinations and generate proposals outside \(X\times Y\) while ensuring \(z_{t+1}\in X\times Y\).
After a restart, we set $z_\star\leftarrow z_{t+1}$, $\mathbf K_{t+1}\leftarrow0$, and $s\leftarrow0$.

\noindent
\begin{minipage}[t]{0.40\linewidth}
\vspace{0pt}
\raggedright
Table~\ref{tab:ga-pdhg-specializations} shows how coefficient choices in \Eqref{eq:gen_acc_iter} correspond to five acceleration and relaxation schemes for fixed-point problems.
GA-PDHG thus offers a unified search space that encompasses these schemes as special cases.
For general coefficient choices, Proposition~\ref{prop:ga-pdhg-optimality} gives sufficient conditions under which the output of a convergent GA-PDHG trajectory is primal--dual optimal in the limit.
Note that we do not claim global convergence for arbitrary learned control sequences.
See Appendix~\ref{app:proofs} for details.
\end{minipage}\hfill
\begin{minipage}[t]{0.57\linewidth}
\vspace{0pt}
\centering
\small
\captionsetup{type=table,font=small,position=top,justification=raggedright,singlelinecheck=false,hypcap=false,skip=5pt}
\caption{Acceleration and relaxation schemes for fixed-point problems via GA-PDHG coefficients.}
\label{tab:ga-pdhg-specializations}
\setlength{\tabcolsep}{1pt}
\renewcommand{\arraystretch}{1.0}
\begin{tabularx}{\linewidth}{@{}>{\raggedright\arraybackslash}Xccc@{}}
\toprule
Method & $\alpha_t^x\!=\!\alpha_t^y\!=\!\alpha_t$ & $\boldsymbol\beta_t$ & $\zeta_t$\\
\midrule
Halpern~\citep{halpern1967fixed} & $0$ & $\mathbf0$ & $\frac{1}{s+2}$\\
Reflected Halpern~\citep{lu2024halpern} & $\frac{s+1}{s+2}$ & $\mathbf0$ & $\frac{1}{s+2}$\\
Anderson~\citep{anderson1965iterative} & $0$ & $\boldsymbol\beta_t^{\mathrm{AA}}$ & $0$\\
\mbox{Krasnosel'ski\u\i--Mann (KM)} \mbox{\citep{krasnosel1955two,mann1953mean}} & $\lambda-1$ & $\mathbf0$ & $0$\\
Fast KM~\citep{bot2023fastkm} & $\frac{\nu}{2\left(s+\nu\right)}-1$ & $\frac{s}{s+\nu}$ & $0$\\
\bottomrule
\end{tabularx}
\par\vspace{4pt}
\footnotesize
\raggedright
$\boldsymbol\beta_t^{\mathrm{AA}}$ is obtained by solving a least squares problem. $\lambda\!\in\!(0,2)$. 
For Fast KM, $k=1$, $\nu>2$ and $s\geq1$.
\end{minipage}
\par

\begin{proposition}[Optimality of convergent GA-PDHG]
\label{prop:ga-pdhg-optimality}
Suppose $\boldsymbol\beta_t$ is bounded, $\eta_t,\omega_t$ are uniformly positive and bounded, and there exists $\delta>0$ such that $1+\alpha_t^x,1+\alpha_t^y\geq\delta$ whenever $e_t=0$.
If both $\{\hat z_t\}$ and $\{z_t\}$ converge to finite limits and $\left\lVert\zeta_t\left(z_\star-z_t\right)\right\rVert_2\longrightarrow0$ along the trajectory, then the limit of $\{z_t\}$ is a primal--dual optimal solution of \Eqref{eq:LP}.
\end{proposition}

Selecting the coefficients of GA-PDHG can be viewed as an optimal control problem in which the GA-PDHG iteration defines a discrete-time dynamical system and the objective is to accelerate convergence.
Figure~\ref{fig:framework} illustrates a feedback control loop in which the solver inputs are selected based on the current solver state. At iteration $t$, the solver state $S_t$ contains the LP data, input--output pair $(z_t, \hat z_t)$, anchor $z_\star$, and history $\mathbf K_t$, the previous control variables, convergence records, and iteration counters. 
From this state, we construct an observation $o_t=\operatorname{Obs}\left(S_t\right)$ whose dimension is independent of the problem size. 
A policy network inputs $o_t$ and outputs an action $a_t$ comprising the restart decision $e_t$ and GA-PDHG coefficients $\alpha_t^x,\alpha_t^y,\boldsymbol\beta_t,\zeta_t$, and PDHG parameters $\omega_t,\eta_t$. 
Using this action, the solver performs the GA-PDHG update and updates its diagnostics and counters to obtain $S_{t+1}$. 
The next observation $o_{t+1}=\operatorname{Obs}\left(S_{t+1}\right)$ is then constructed and passed to the policy for the next decision.
\begin{figure}[!htbp]
    \centering
    \includegraphics[width=0.95\linewidth]{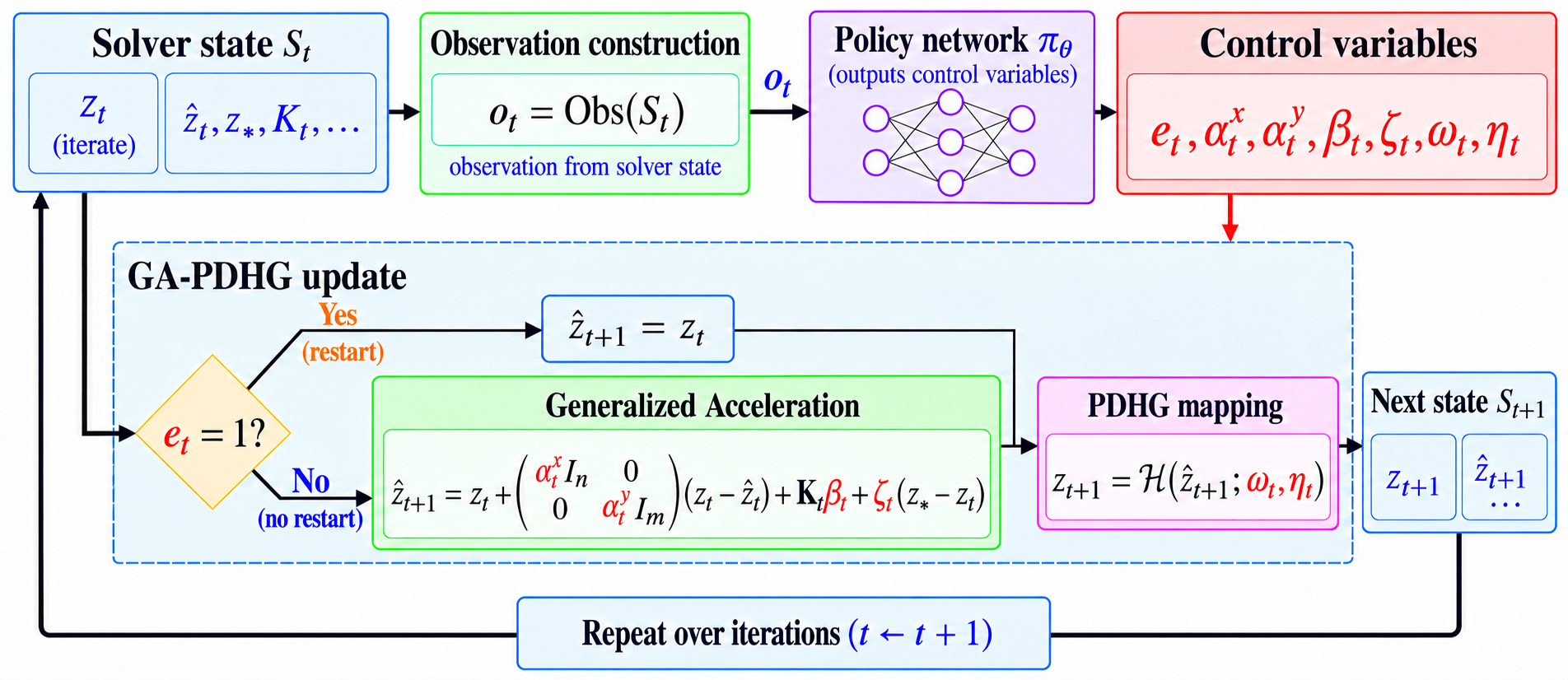}
    \caption{Feedback control loop for GA-PDHG. Red symbols denote control variables selected by the policy network from solver observations.}
    \label{fig:framework}
\end{figure}
\vspace{-20pt}
\subsection{Reinforcement Learning Formulation}
\label{sec:learning-controller}
We use reinforcement learning (RL) to train the feedback policy $\pi_\theta$ for GA-PDHG and refer to the resulting solver as \textbf{GALLOP}.
Each LP solve corresponds to an episode, during which the policy selects solver controls at each iteration to maximize the expected cumulative reward.
To support size generalization, we design the observation and control spaces to have dimensions independent of $n$ and $m$, allowing the same policy to be evaluated on LPs of different sizes.
At each decision step, the policy receives the observation $o_t = \operatorname{Obs}(S_t)$, which summarizes the solver state.
Log-normalized primal and dual residuals, the duality gap, their changes, and fixed-point residual progress describe current error and recent improvement.
Iteration count, restart age, progress from the anchor, primal--dual displacement balance, relative primal weight, and its one-step change characterize the current epoch.
A normalized history Gram matrix, inner products with the PDHG direction, and anchor magnitude and alignment describe the scales and relative directions of the available corrections.
Given this observation, the policy selects controls in a conditional hybrid action space, in which the discrete restart decision determines which continuous controls affect the solver transition.
The policies for $e_t$, $\eta_t$, and $\omega_t$ are trained using both restart and non-restart transitions, whereas the policies for the acceleration coefficients $\left(\alpha_t^x,\alpha_t^y,\boldsymbol\beta_t,\zeta_t\right)$ are trained only on non-restart transitions.

We measure the quality of iterate $z\in X\times Y$ in original LP coordinates using the primal residual, dual residual, and duality-gap measure.
Let $\bar c:=c-A^\top y$ be the reduced cost and $\lambda:=\operatorname{proj}_{\Lambda}\left(\bar c\right)$ its projection onto $\Lambda$.
The primal and dual residuals and the duality-gap measure are defined as
\begin{equation}
\label{eq:residuals}
  r_p\left(z\right) := \operatorname{proj}_Y\left(b-Ax\right),
  \quad r_d\left(z\right) := \bar c-\lambda,
  \quad r_g\left(z\right) := c^\top x-\left(b^\top y+\min_{x\in X}\lambda^\top x\right).
\end{equation}
We denote the logarithms of their normalized magnitudes by $p\left(z\right)$, $d\left(z\right)$, and $g\left(z\right)$ and define the merit function as
\begin{equation}
\label{eq:controller-merit}
    M^{\mathrm{rel}}\left(z\right):=\max\left\{p\left(z\right),d\left(z\right),g\left(z\right)\right\}.
\end{equation}
For the target tolerance $\epsilon_{\mathrm{tol}}$, an episode ends successfully when the convergence criterion $M^{\mathrm{rel}}\left(z_{t+1}\right)\leq\log\epsilon_{\mathrm{tol}}$ is met. If the episode does not converge, it is truncated after $T_{\max}$ controlled iterations.
With $[a]^+:=\max\{a,0\}$, we use the following reward function to encourage reaching tolerance in fewer iterations,
\begin{equation}
\label{eq:reward}
\begin{aligned}
    r_t\!=\!M^{\mathrm{rel}}\!\left(z_t\right)\!-\!M^{\mathrm{rel}}\!\left(z_{t+1}\right)
    \!-\!\chi\left[M^{\mathrm{rel}}\!\left(z_{t+1}\right)\!-\!\log\epsilon_{\mathrm{tol}}\right]^{\!+} \!\!\!-\!\frac{1}{T_{\max}}
    \!+\!\mathbf 1\!\left[M^{\mathrm{rel}}\!\left(z_{t+1}\right)\!\leq\!\log\epsilon_{\mathrm{tol}}\right]\!.
\end{aligned}
\end{equation}
Here, $\chi\geq0$ weights the excess-error penalty $\left[M^{\mathrm{rel}}\!\left(z_{t+1}\right)\!-\!\log\epsilon_{\mathrm{tol}}\right]^{\!+}$ and $\mathbf1[\cdot]$ is the indicator function.
The difference in merit $M^{\mathrm{rel}}\!\left(z_t\right)\!-\!M^{\mathrm{rel}}\!\left(z_{t+1}\right)$ rewards a reduction in KKT error from one iteration to the next, even when the error remains above the target tolerance.
The excess-error penalty encourages reducing errors earlier along the trajectory.
Summed over an episode without discounting, the last two terms contribute $1-T/T_{\max}$ for success after $T$ iterations and $-1$ for truncation without success, favoring successful solves that reach tolerance in fewer iterations.

\subsection{Policy Learning with Proximal Policy Optimization}
\label{sec:policy-training}

We design a groupwise proximal policy optimization (PPO) objective that reflects both conditional activation and the distinct control roles in GA-PDHG.
We assign actions to five groups, $\mathcal G=\{e,\beta,(\alpha,\zeta),\eta,\omega\}$, with a Bernoulli policy for restart and Gaussian policies for the continuous controls.
For trajectories collected under fixed policy parameters $\theta_{\mathrm{old}}$, let $a_t^j$ be the recorded action of group $j\in\mathcal G$, and let $\pi^j_{\theta}(a_t^j|o_t)$ denote its probability mass or density under policy parameters $\theta$.
On a transition where group $j$ is active, its likelihood ratio is $\rho_t^j :=
    \frac{\pi^j_{\theta}\;(a_t^j|o_t)}
    {\pi^j_{\theta_{\mathrm{old}}}(a_t^j|o_t)}$.
On restart transitions, we marginalize the unused acceleration controls in groups $\beta$ and $(\alpha,\zeta)$ out of the likelihood and set $\rho_t^j\equiv1$ for these inactive groups, so they contribute no policy gradient.
A shared state-value critic $V^{\phi}\left(o_t\right)$ provides value estimates for generalized advantage estimation (GAE)~\citep{schulman2016gae}.
Let $\tilde\Psi_t$ denote the resulting advantage normalized over the collected transitions.
All groups use this common advantage to learn from the return of their joint execution.
The objective of our groupwise PPO is
\begin{equation}
\label{eq:ppo-policy-objective}
    \mathcal J \left(\theta\right)
    =
    \sum_{j\in\mathcal G}
    \hat{\mathbb E}_t
    \left[
        \min\left(
            \rho_t^j\tilde\Psi_t,\,
            \clip\left(
                \rho_t^j,
                1-\varepsilon^{\mathrm{clip}},
                1+\varepsilon^{\mathrm{clip}}
            \right)\tilde\Psi_t
        \right)
    \right].
\end{equation}
Here, $\hat{\mathbb E}_t$ averages over the collected transitions. $\varepsilon^{\mathrm{clip}}>0$ is the clipping parameter, which is applied to each group's likelihood ratio $\rho_t^j$.
For the standard PPO~\citep{schulman2017ppo}, clipping is applied to the joint action likelihood ratio $\prod_{j\in\mathcal G}\rho_t^j$.
Under the joint clipping, changes in one group's likelihood can suppress the gradient for the other group's policy even when the the other group's likelihood ratio is close to one.
In particular, restart probability is often below $10\%$, hence a small logit change has little effect on the non-restart likelihood. On the other hand $\beta$ is sampled from $k$-dimensional Gaussian, which accumulates coordinate contributions in its log-likelihood ratio.
\begingroup
\setlength{\intextsep}{2pt}
\setlength{\columnsep}{10pt}
\Needspace{5\baselineskip}
\begin{wraptable}{r}{0.25\linewidth}
\begin{minipage}[!h]{\linewidth}
\centering
\small
\captionsetup{type=table,font=small,position=top,justification=centering,singlelinecheck=false,hypcap=false,skip=5pt}
\caption{Clipping ratio (\%).}
\label{tab:groupwise-clipping}
\setlength{\tabcolsep}{2pt}
\begin{tabular}{@{}l|c|c@{}}
\toprule
Problem class & $\beta$ & $e$\\
\midrule
Transport & 3.79 & 0.35\\
Auction & 8.99 & 0.94\\
Random LPs & 2.01 & 0.13\\
\bottomrule
\end{tabular}
\end{minipage}
\end{wraptable}
\noindent
Table~\ref{tab:groupwise-clipping} shows that clipping of $\beta$ happens more frequently than restart $e$, motivating separate clipping decisions.
Under our groupwise PPO objective in \Eqref{eq:ppo-policy-objective}, each group's clipping depends only on its own ratio, so a change in one group's ratio cannot by itself clip another group's ratio.
Similar discussions can be found in ~\citet{fan2019hybrid}. 
At inference, we fix $\theta$, threshold the restart probability at $0.5$, and transform the Gaussian means into continuous controls.
\textcolor{black}{GALLOP} then executes the feedback loop in Figure~\ref{fig:framework} until the KKT tolerance is reached or the iteration limit is exhausted.
The critic is used only during training. See Appendix~\ref{app:implementation-details} for details.
\finishwrap\endgroup
Off-policy actor--critic methods, such as SAC~\citep{haarnoja2018sac}, can also be extended to hybrid action spaces~\citep{delalleau2019hybrid}. However, these methods require the hybrid structure to be represented in both the policy and the action-value critic, including the treatment of inactive continuous controls.
Because environment interaction is inexpensive in our setting, the sample-efficiency gains from off-policy data reuse offer limited benefit.

\Needspace{18\baselineskip}
\section{Experiments}
\begingroup
\setlength{\parskip}{4pt}
\setlength{\intextsep}{8pt plus 2pt minus 2pt}
\setlength{\textfloatsep}{10pt plus 2pt minus 2pt}

\subsection{Experiment Setting}
\label{sec:experimental-setup}
\begingroup
\setlength{\intextsep}{4pt}
\setlength{\columnsep}{14pt}
\Needspace{14\baselineskip}
\leavevmode
\begin{wraptable}{r}{0.35\linewidth}
\begin{minipage}[t]{\linewidth}
\centering
\footnotesize
\captionsetup{type=table,font=small,justification=raggedright,singlelinecheck=false,hypcap=false,skip=5pt}
\setlength{\tabcolsep}{0pt}
\renewcommand{\arraystretch}{0.9}
\caption{Density $\operatorname{nnz}\left(A\right)/(mn)$ and $m/n$ at each family's largest training size.}
\label{tab:lp-family-comparison}
\begin{tabular*}{\linewidth}{@{\extracolsep{\fill}}lrr@{}}
\toprule
Family & Density & $m/n$\\
\midrule
Random LP & $1.28\%$ & $0.5$\\
Transport & $0.625\%$ & $0.0125$\\
Auction & $0.313\%$ & $0.25$\\
Min-cost flow & $0.0417\%$ & $0.167$\\
Robust inventory & $0.326\%$ & $0.341$\\
$L_\infty$ regression & $100\%$ & $26.7$\\
\bottomrule
\end{tabular*}
\end{minipage}
\end{wraptable}
\noindent
We evaluate our method on six different LP families in Table~\ref{tab:lp-family-comparison}; Random LP, Transport, Auction, Min-cost flow, Robust inventory management, and $L_\infty$ regression.
We train one policy per family and evaluate on unseen LPs with varying sizes within the same family.
Appendix~\ref{app:lp-problem-classes} gives the problem details and size lists.
We compare GALLOP with MPAX's $\mathrm{r}^2$HPDHG solver~\citep{lu2024mpax} and GA-PDHG (HT), a hand-tuned baseline that shares our solver implementation.
GA-PDHG (HT) uses the reflected Halpern coefficients in Table~\ref{tab:ga-pdhg-specializations}, and $\eta_t=0.998/\lVert A\rVert_2$, with restarting rules and PID controller weight updates from cuPDLPx~\citep{lu2025cupdlpx}.
Because MPAX checks convergence every 100 iterations by default whereas our implementation checks every iteration, we use saved trajectories to identify the first iterate satisfying $M^{\mathrm{rel}}\left(z_t\right)\leq\log\epsilon_{\mathrm{tol}}$, with $\epsilon_{\mathrm{tol}}=10^{-6}$.
We report iterations and matrix--vector products to this common tolerance, including products from MPAX's rejected adaptive steps.
Algorithm wall clock measures GPU execution of the main solve loop after compilation and excludes JIT compilation and preconditioning.

\finishwrap\endgroup

\subsection{Convergence and Size Generalization}
\label{sec:convergence-generalization}

Figure~\ref{fig:synthetic-residuals} plots the exponentiated merit \(\exp (M^{\mathrm{rel}}(z_t))\) against iteration count.
\textcolor{black}{GALLOP} reaches the tolerance $10^{-6}$ in fewer iterations than both baselines on all six plotted instances.
Table~\ref{tab:main-results} reports medians over four unseen LPs per family and size.

\begin{figure}[!ht]
    \centering
    \captionsetup{font=small}
    \captionsetup[subfigure]{skip=2pt}
    \includegraphics[width=0.85\linewidth]{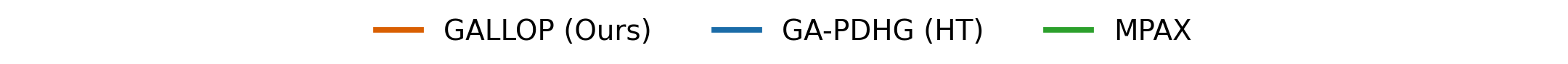}

    \vspace{1pt}
    \begin{subfigure}[t]{0.331\linewidth}
        \centering
        \includegraphics[width=\linewidth]{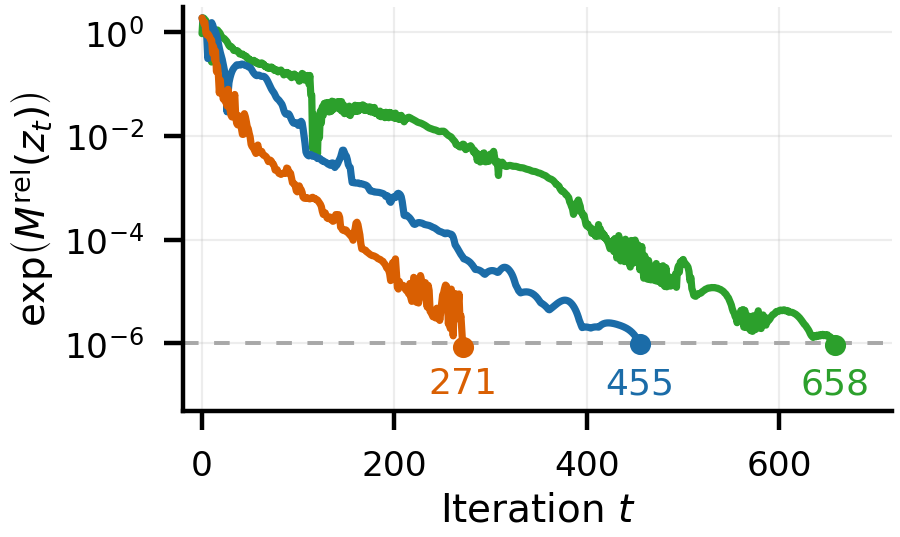}
        \caption{Random LP}
        \label{fig:synthetic-residual-random}
    \end{subfigure}\hfill
    \begin{subfigure}[t]{0.331\linewidth}
        \centering
        \includegraphics[width=\linewidth]{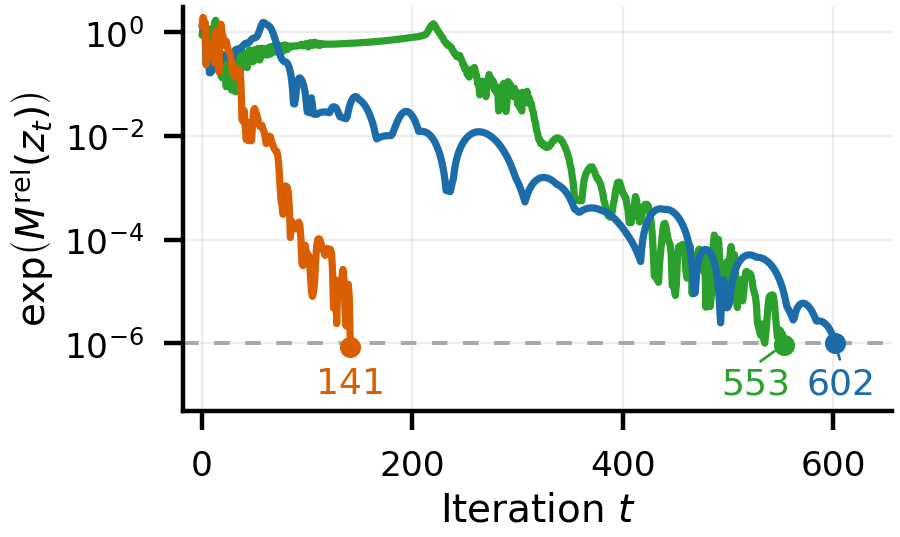}
        \caption{Transport}
        \label{fig:synthetic-residual-transport}
    \end{subfigure}\hfill
    \begin{subfigure}[t]{0.331\linewidth}
        \centering
        \includegraphics[width=\linewidth]{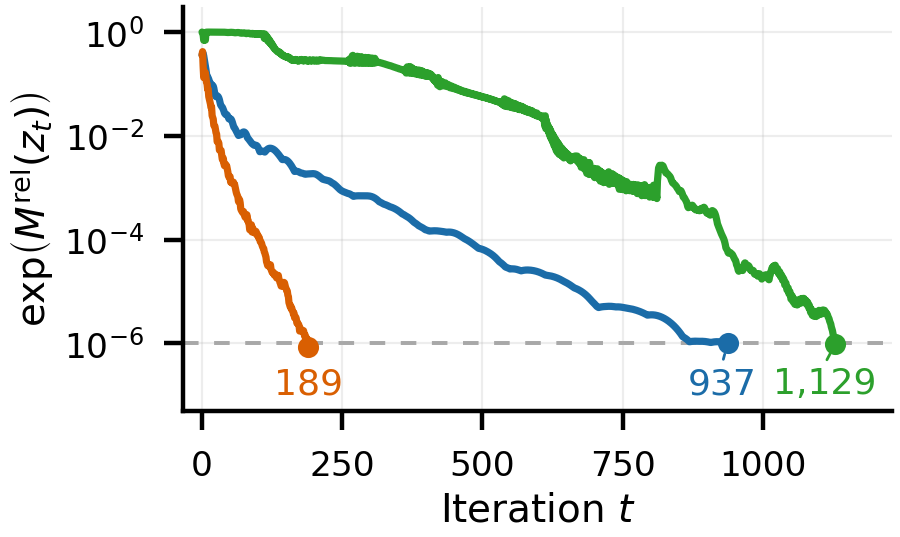}
        \caption{Auction}
        \label{fig:synthetic-residual-auction}
    \end{subfigure}

    \vspace{3pt}
    \begin{subfigure}[t]{0.331\linewidth}
        \centering
        \includegraphics[width=\linewidth]{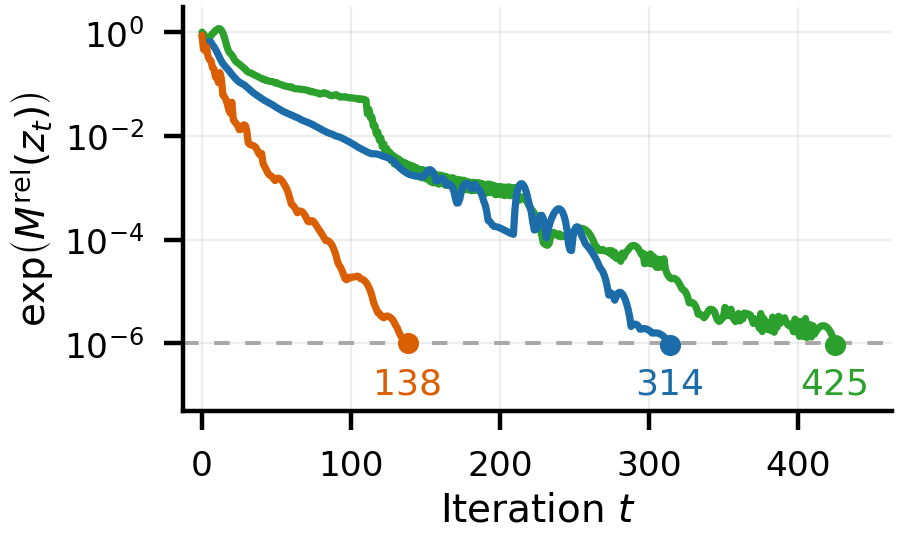}
        \caption{Min-cost flow}
        \label{fig:synthetic-residual-mcf}
    \end{subfigure}\hfill
    \begin{subfigure}[t]{0.331\linewidth}
        \centering
        \includegraphics[width=\linewidth]{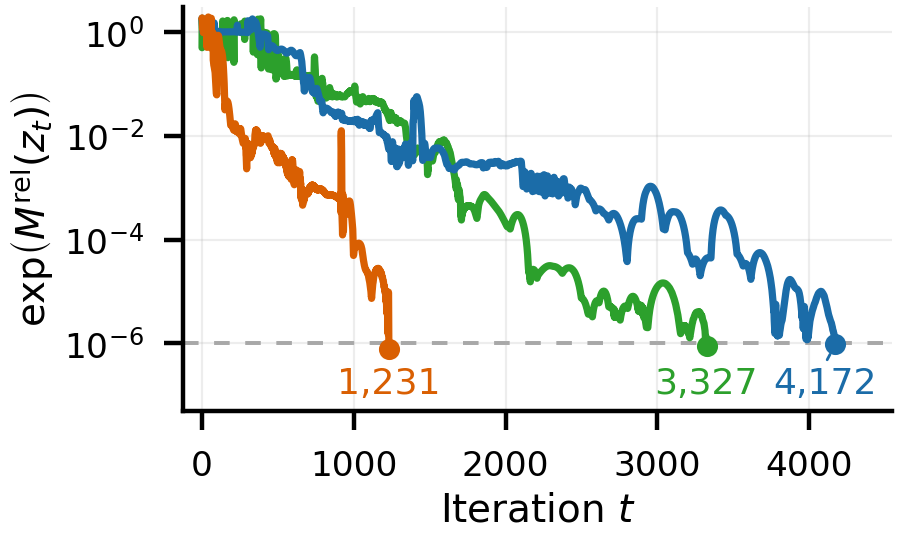}
        \caption{Robust inventory}
        \label{fig:synthetic-residual-inventory}
    \end{subfigure}\hfill
    \begin{subfigure}[t]{0.331\linewidth}
        \centering
        \includegraphics[width=\linewidth]{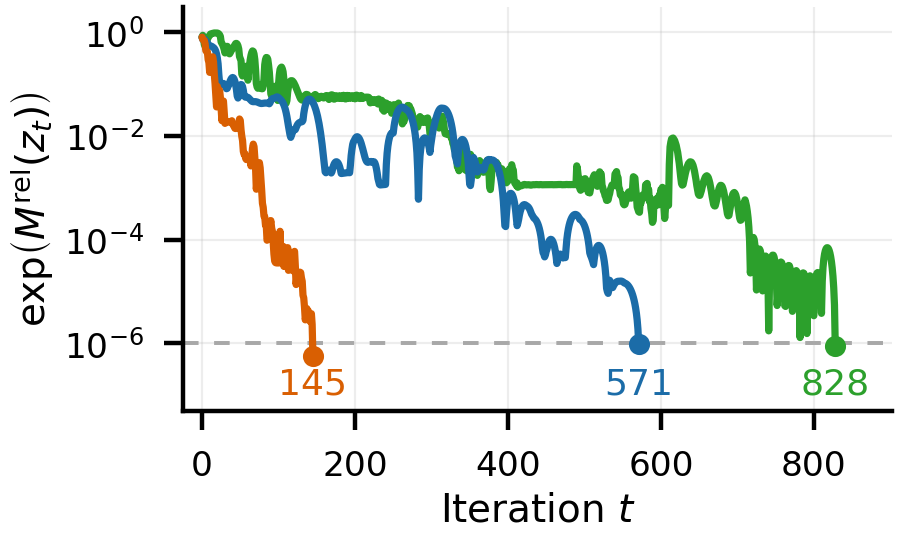}
        \caption{$L_\infty$ regression}
        \label{fig:synthetic-residual-regression}
    \end{subfigure}
    \caption{Convergence on one unseen LP per family at the largest training size. Labels mark the first iteration meeting tolerance $10^{-6}$ (dashed line).}
    \label{fig:synthetic-residuals}
\end{figure}
 
The first row for each family corresponds to the largest training size ($1\times$), and the following rows evaluate the same policy on larger LPs without retraining.
At $1\times$, \textcolor{black}{GALLOP} requires fewer iterations and matrix--vector products for all six families.
\textcolor{black}{GALLOP} generalizes to $15\times$ for Random LP, $400\times$ for Transport, $14\times$ for Auction, $12\times$ for Min-cost flow, $3\times$ for Robust inventory, and $7\times$ for $L_\infty$ regression, with gains in all three metrics at these sizes.
Size multipliers are measured by the number of variables, except for Robust inventory (item count) and $L_\infty$ regression (constraint count at fixed $n=9$).

\begin{table}[!t]
\centering
\fontsize{8}{9.5}\selectfont
\setlength{\tabcolsep}{0pt}
\renewcommand{\arraystretch}{1.0}
\captionsetup{font=small,skip=5pt}
\caption{Test results (four-LP medians). The “Size ratio” column reports the ratio of test size to each family's largest training size. Size is measured in $n$ except for robust inventory (item count $N$) and $L_\infty$ regression ($m$ at fixed $n=9$). Ours denotes \textcolor{black}{GALLOP}, HT denotes GA-PDHG (HT).}
\label{tab:main-results}
\begin{tabular*}{\linewidth}{@{\extracolsep{\fill}}llc rrr rrr rrr@{}}
\toprule
& & & \multicolumn{3}{c}{Iterations} & \multicolumn{3}{c}{\shortstack{Matrix--vector\\products}} & \multicolumn{3}{c}{\shortstack{Algorithm wall\\clock (ms)}}\\
\cmidrule(lr){4-6}\cmidrule(lr){7-9}\cmidrule(l){10-12}
Family & $(m,n)$ & Size ratio & Ours & HT & MPAX & Ours & HT & MPAX & Ours & HT & MPAX\\
\midrule
Random LP & (10,000, 20,000) & $1\times$ & \textbf{284} & 440 & 664 & \textbf{570} & 883 & 1,384 & \textbf{110} & 140 & 408\\
 & (30,000, 60,000) & $3\times$ & \textbf{280} & 421 & 663 & \textbf{561} & 844 & 1,384 & \textbf{151} & 199 & 1,175\\
 & (150,000, 300,000) & $15\times$ & \textbf{290} & 428 & 702 & \textbf{582} & 859 & 1,468 & \textbf{450} & 626 & 7,202\\
\addlinespace[2pt]
Transport & (320, 25,600) & $1\times$ & \textbf{164} & 633 & 548 & \textbf{329} & 1,268 & 1,150 & \textbf{39} & 107 & 87\\
 & (1,600, 640,000) & $25\times$ & \textbf{153} & 533 & 536 & \textbf{308} & 1,068 & 1,128 & \textbf{102} & 286 & 308\\
 & (6,400, 10,240,000) & $400\times$ & \textbf{200} & 570 & 524 & \textbf{403} & 1,142 & 1,105 & \textbf{1,998} & 5,222 & 7,334\\
\addlinespace[2pt]
Auction & (6,400, 25,600) & $1\times$ & \textbf{190} & 928 & 980 & \textbf{381} & 1,859 & 2,025 & \textbf{70} & 256 & 198\\
 & (32,000, 128,000) & $5\times$ & \textbf{190} & 923 & 994 & \textbf{382} & 1,848 & 2,055 & \textbf{84} & 316 & 479\\
 & (89,600, 358,400) & $14\times$ & \textbf{192} & 926 & 1,066 & \textbf{387} & 1,853 & 2,198 & \textbf{134} & 521 & 1,352\\
\addlinespace[2pt]
Min-cost flow & (4,800, 28,800) & $1\times$ & \textbf{164} & 302 & 440 & \textbf{329} & 606 & 928 & \textbf{49} & 68 & 76\\
 & (24,000, 144,000) & $5\times$ & \textbf{234} & 356 & 444 & \textbf{469} & 715 & 939 & \textbf{80} & 94 & 109\\
 & (57,600, 345,600) & $12\times$ & \textbf{242} & 370 & 482 & \textbf{485} & 741 & 1,010 & \textbf{115} & 142 & 147\\
\addlinespace[2pt]
Robust inventory & (1,500, 4,401) & $1\times$ & \textbf{1,905} & 4,266 & 3,875 & \textbf{3,812} & 8,535 & 7,872 & 507 & 915 & \textbf{432}\\
 & (2,980, 8,801) & $2\times$ & \textbf{2,311} & 7,768 & 5,348 & \textbf{4,624} & 15,537 & 10,854 & \textbf{570} & 1,494 & 661\\
 & (4,460, 13,201) & $3\times$ & \textbf{2,874} & 25,798 & 6,121 & \textbf{5,749} & 51,599 & 12,391 & \textbf{858} & 5,978 & 895\\
\addlinespace[2pt]
$L_\infty$ regression & (240, 9) & $1\times$ & \textbf{150} & 586 & 834 & \textbf{303} & 1,175 & 1,734 & \textbf{24} & 54 & 82\\
 & (1,200, 9) & $5\times$ & \textbf{236} & 673 & 760 & \textbf{475} & 1,348 & 1,583 & \textbf{41} & 78 & 69\\
 & (1,680, 9) & $7\times$ & \textbf{358} & 752 & 782 & \textbf{718} & 1,505 & 1,626 & \textbf{62} & 93 & 78\\
\bottomrule
\end{tabular*}
\end{table}

\FloatBarrier
\clearpage
\begingroup
\setlength{\intextsep}{4pt}
\setlength{\columnsep}{12pt}
\begin{wrapfigure}{R}{0.47\linewidth}
\begin{minipage}[t]{\linewidth}
\centering
\captionsetup{type=figure,font=small,justification=raggedright,singlelinecheck=false,hypcap=false,skip=4pt}
\includegraphics[width=\linewidth]{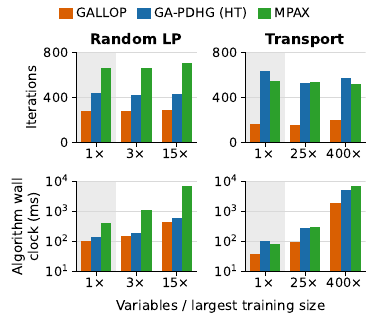}
\caption{Size generalization (four-LP medians). Shading marks the largest training size.}
\label{fig:synthetic-wall-clock}

\par\vspace{8pt}
\small
\captionsetup{type=table,font=small,position=top,justification=raggedright,singlelinecheck=false,hypcap=false,skip=4pt}
\caption{Transport control ablations (four-LP medians, without retraining).}
\label{tab:head-ablation}
\setlength{\tabcolsep}{2.5pt}
\begin{tabular}{@{}llr@{}}
\toprule
Control group & Replaced by & Iterations\\
\midrule
\multicolumn{2}{@{}l}{--- (\textcolor{black}{GALLOP} as trained)} & \textbf{165.0}\\
\addlinespace[2pt]
Step size $\eta_t$ & $0.998/\lVert A\rVert_2$ & 457.0\\
Primal weight $\omega_t$ & PID rule \Eqref{eq:pdlp-weight-background} & 224.5\\
$\alpha^x_t,\alpha^y_t,\zeta_t$ & Halpern & 209.0\\
History $\boldsymbol\beta_t$ & $\boldsymbol\beta_t=\mathbf 0$ & 190.5\\
History $\boldsymbol\beta_t$ & Anderson & 908.0\\
Restart flag $e_t$ & \begin{tabular}[t]{@{}l@{}}three conditions\\\Eqref{eq:pdlp_restart_condition}\end{tabular} & 181.0\\
\addlinespace[2pt]
All groups & GA-PDHG (HT) & 633.0\\
\bottomrule
\end{tabular}

\end{minipage}
\end{wrapfigure}
\noindent
Figure~\ref{fig:synthetic-wall-clock} highlights the corresponding algorithm wall clock savings for random LP and Transport.
At $400\times$, Transport takes 1,998\,ms for \textcolor{black}{GALLOP}, compared with 5,222\,ms for GA-PDHG (HT) and 7,334\,ms for MPAX.
Tables~\ref{tab:ladder-results-a} and~\ref{tab:ladder-results-b} provide results for all test sizes.

\subsection{Ablation study}
\label{sec:ablation}
We evaluate the contribution of each learned control group on Transport at the largest training size, $m=320$ and $n=25{,}600$.
At evaluation time, we replace one control group with a hand-tuned rule and retain the other learned controls without retraining.
We also test Anderson acceleration in place of the learned history coefficients.
Table~\ref{tab:head-ablation} reports the median number of iterations to reach $\epsilon_{\mathrm{tol}}=10^{-6}$ on all four LPs.
Among substitutions with hand-tuned rules, replacing the learned step size has the largest effect, increasing the iterations from 165 to 457.
Replacing any of the other four groups increases the median by $10\%$--$36\%$.
These results support the effectiveness of GA-PDHG and retaining all five learned control groups in the trained policy.
Since the remaining controls are kept fixed rather than retrained, this experiment measures reliance of the learned joint policy on each control group, rather than the best performance achievable after removing that group and retraining.
\finishwrap\endgroup

\subsection{Comparison on Public dataset}
\label{sec:public-dataset}

\begingroup
\setlength{\intextsep}{4pt}
\setlength{\columnsep}{10pt}
\Needspace{20\baselineskip}
\leavevmode
\begin{wraptable}{r}{0.29\linewidth}
\begin{minipage}[t]{\linewidth}
\centering
\small
\captionsetup{type=table,font=small,position=top,justification=raggedright,singlelinecheck=false,hypcap=false,skip=4pt}
\caption{IP-S. Mean iteration counts.}
\label{tab:ip-s-results}
\setlength{\tabcolsep}{0pt}
\begin{tabular*}{\linewidth}{@{\extracolsep{\fill}}lr@{}}
\toprule
Solver & Iterations\\
\midrule
GALLOP & \textbf{186.7}\\
MPAX & 1,089.0\\
PDHG-Net + PDLP & 422\\
PDLP & 550.4\\
\bottomrule
\end{tabular*}
\par\vspace{6pt}
\caption{Comparison results on IP-L. Zero-shot evaluation.}
\label{tab:ip-l-results}
\setlength{\tabcolsep}{0pt}
\begin{tabular*}{\linewidth}{@{\extracolsep{\fill}}lr@{}}
\toprule
Solver & Iterations\\
\midrule
GALLOP & \textbf{272.2}\\
MPAX & OOM\\
PDLP & 851.2\\
\bottomrule
\end{tabular*}
\par\vspace{4pt}
{\footnotesize\raggedright OOM: out of memory on  RTX~4090.\par}
\end{minipage}
\end{wraptable}
\noindent
We evaluate on the public item-placement benchmark from \citet{Li2024PDHGUnrolled}.
IP-S has $\left(m,n\right)=\left(15{,}525,31{,}350\right)$ and $\operatorname{nnz}\left(A\right)=5{,}291{,}250$, while IP-L has $\left(m,n\right)=\left(91{,}575,266{,}450\right)$ and $\operatorname{nnz}\left(A\right)=94{,}826{,}250$.

We train GALLOP on 128 IP-S instances and evaluate the same policy on ten unseen instances of IP-S and IP-L. IP-L evaluation tests zero-shot size generalization without retraining.

Tables~\ref{tab:ip-s-results} and~\ref{tab:ip-l-results} report mean iteration counts for each solvers.
GALLOP and MPAX use tolerance $10^{-6}$. PDLP uses the released PDHG-Net settings with absolute tolerance $10^{-6}$ and relative tolerance $10^{-4}$. This stopping criterion is looser than GALLOP's. PDLP is the OR-Tools~9.8.3296 version run on our machine.
The IP-S result of PDHG-Net is quoted from Table 2 in \citet{Li2024PDHGUnrolled}, which predicts a initialization for PDLP, hence the label PDHG-Net + PDLP.

On IP-S, GALLOP requires $5.83\times$ and $2.95\times$ less number of iterations relative to MPAX and PDLP, respectively, and uses fewer iterations than the reported PDHG-Net + PDLP result.
On IP-L, the reduction relative to PDLP is a factor of $3.13\times$.
Its iteration count increases only $1.46\times$ despite $8.50\times$ more variables and $17.92\times$ more nonzeros, demonstrating size generalization on this benchmark.

\finishwrap\endgroup

\section{Related Work}
\label{sec:related-work}
\textbf{RL for discrete-continuous hybrid action spaces.}
P-DQN~\citep{xiong2018pdqn} and MP-DQN~\citep{bester2019mpdqn} use action-value critics, unlike GALLOP's state-value critic.
H-PPO~\citep{fan2019hybrid} also uses a state-value critic and separately clips actor objectives; DISC~\citep{han2019disc} instead clips coordinatewise importance ratios.
GALLOP specializes actor-wise clipping to five functional solver-control groups, marginalizing acceleration controls on restart transitions while retaining step-size and primal-weight updates.

\textbf{L2O for LP.}
Prior work learns simplex initial bases~\citep{fan2023smart}, imitates pivot experts~\citep{liu2024pivot}, or emulates interior-point methods~\citep{qian2024exploring}.
These approaches target initialization or expert-derived decisions rather than jointly learning PDHG acceleration and restarting.
More directly related, PDHG-Net~\citep{Li2024PDHGUnrolled} unrolls PDHG to warm-start PDLP.
GALLOP instead retains learned feedback throughout the solve, without handing subsequent iterations to an external solver.

\textbf{RL in L2O.}
RL learns simplex pivot-rule selection~\citep{suriyanarayana2022rl} and presolve routines~\citep{kuang2025presolve}, addressing discrete solver decisions rather than hybrid iterative control.
RLQP~\citep{Ichnowski2021Accelerating} adapts continuous penalty parameters of the alternating direction method of multipliers (ADMM), whereas TFPnP~\citep{wei2020tfpnp} combines continuous parameter tuning with discrete termination for inverse imaging.
GALLOP instead jointly controls PDHG step sizes, primal weights, generalized acceleration, and restarts through a dimension-agnostic policy.
Its discrete decisions reset acceleration state rather than terminate optimization.

\endgroup

\section{Conclusion and Limitations}
\label{sec:conclusion}

We introduced GALLOP, which uses reinforcement learning to jointly control PDHG parameters, acceleration coefficients, and restart decisions.
Its GA-PDHG update combines primal and dual extrapolation, history corrections, and an anchor correction in a common action space.
A dimension-agnostic policy enables deployment on larger LPs without retraining.
Across six LP families, GALLOP reduces iteration counts relative to MPAX and achieves algorithm wall-clock speedups of up to $16.0\times$.
Policies trained on smaller LPs generalize to sizes up to $400\times$ larger, with additional evidence of size generalization on the public item-placement benchmark.

\textbf{Limitations.} Available GPU memory limited the maximum LP sizes evaluated in this study.
Experiments with more GPU memory would allow us to test whether the observed size generalization extends to substantially larger LPs.

\FloatBarrier

\clearpage
\appendix
\section*{Appendix}

\section{Extended Related Work}
\label{app:extended-related-work}

\paragraph{Learning to Optimize and Deep unfolding.}
Learning to optimize (L2O) leverages data-driven machine learning in order to develop and improve optimization methods.
Key survey papers include \citep{Chen2022Learning} and \citep{Amos2023Tutorial}.
Deep unfolding is a model-based L2O approach that interprets each iteration of an optimizer as a layer of a neural network~\citep{Hershey2014Deep}.
Learnable parameters are introduced and the optimizer is unrolled for a fixed number of iterations and trained using backpropagation-through-time, much like a recurrent neural network~\citep{Monga2021Algorithm}.
The idea has seen great success in the signal and image processing communities after the seminal work by \citet{Gregor2010Learning} on learned iterative shrinkage and soft-thresholding (LISTA).
These approaches focus on learning iteration-varying parameters that can accelerate the optimizer, amortize the cost of tuning, and improve generalization on unseen instances~\citep{Sambharya2024Learning,Amos2023Tutorial}.
Deep unfolding has been recently extended to general classes of convex optimization problems, such as linear programming~\citep{Li2024PDHGUnrolled}, quadratic programming~\citep{saravanos2025deep,Oshin2026Deep}, and conic programming~\citep{Oshin2026Scalable}.
Another line of recent work has studied guarantees of learned optimizers, including PAC-Bayes generalization bounds~\citep{Sucker2025Learning} and worst-case performance bounds~\citep{Sambharya2025Datadriven}.
However, deep unfolding often struggles to scale to high-dimensional problems in machine learning.
This is primarily due to the intractable memory requirement of storing an entire trajectory of optimizer iterations in memory in order to backpropagate through them via supervised learning.
These long-horizon unrolls are also numerically unstable, as they suffer from the vanishing/exploding gradient problem~\citep{Metz2019Understanding}.
Moreover, deep unfolding struggles to learn discrete decisions, such as when to restart in PDLP.
Finally, supervised L2O requires optimal solutions, which are expensive to compute, especially for large-scale problems.
These points make reinforcement learning a more suitable choice for the large-scale LPs that we target in this work.

\section{Proofs from the Paper}
\label{app:proofs}

\subsection{Fixed Points of PDHG and LP Optimality}
\label{app:fixed_point_optimality}
\label{app:pdhg-fixed-point}

We first establish that solving the LP is equivalent to finding a fixed point of the PDHG operator $\mathcal{H}$.
Fix $\omega,\eta>0$ and write $\mathcal H=\mathcal H\left(\cdot;\omega,\eta\right)$. Let $Z^\star$ denote the set of primal--dual optimal solutions of \Eqref{eq:LP}.
For $\alpha\in[0,1]$, define $\mathcal R_\alpha\left(z\right):=\left(1+\alpha\right)\mathcal H\left(z\right)-\alpha z$.

Define the primal--dual KKT mapping $\mathcal{F}:X\times Y\rightarrow 2^{\mathbb{R}^{n+m}}$ by
\begin{equation}
\label{eq:app_kkt_mapping}
    \mathcal{F}(z)
    =
    \mathcal{F}(x,y)
    :=
    \begin{pmatrix}
        c-A^\top y+N_X(x)\\
        Ax-b+N_Y(y)
    \end{pmatrix},
\end{equation}
where $N_X$ and $N_Y$ denote the normal-cone mappings of $X$ and $Y$, respectively.

\begin{lemma}[KKT characterization]
\label{lem:app_kkt}
A primal--dual point $z^\star=(x^\star,y^\star)$ is optimal for \Eqref{eq:LP} if and only if
\begin{equation}
\label{eq:app_kkt_zero}
    0\in\mathcal{F}(z^\star).
\end{equation}
Equivalently,
\begin{align}
    0&\in c-A^\top y^\star+N_X(x^\star),\label{eq:app_kkt_primal}\\
    0&\in Ax^\star-b+N_Y(y^\star),\label{eq:app_kkt_dual}
\end{align}
where \Eqref{eq:app_kkt_dual} is equivalent to
\begin{align}
    Px^\star&=q,\\
    Gx^\star&\ge h,\\
    y^{\mathrm{in},\star}&\ge0,\\
    y^{\mathrm{in},\star}\odot(Gx^\star-h)&=0.
\end{align}
\end{lemma}

\begin{proof}[Proof]
The saddle-point formulation in \Eqref{eq:saddle} is
\begin{equation}
    \min_{x\in X}\max_{y\in Y}
    \left(c^\top x-y^\top Ax+b^\top y\right).
\end{equation}
Its first-order optimality conditions are
\begin{equation}
    0\in c-A^\top y^\star+N_X(x^\star),
    \qquad
    0\in Ax^\star-b+N_Y(y^\star),
\end{equation}
which are necessary and sufficient because the problem is convex--concave and polyhedral.

For the equality component, $y^{\mathrm{eq}}\in\mathbb{R}^{m_e}$ is unrestricted, and therefore the corresponding normal cone is zero. Hence $Px^\star=q$. For the inequality component,
\begin{equation}
    0\in Gx^\star-h+N_{\mathbb{R}^{m_i}_+}(y^{\mathrm{in},\star}),
\end{equation}
which is equivalent to primal feasibility, dual feasibility, and complementary slackness.
\end{proof}

\begin{lemma}[Equivalence between LP solutions and PDHG fixed points]
\label{lem:app_fixed_point}
The primal--dual optimal set coincides with the fixed-point set of the PDHG operator:
\begin{equation}
\label{eq:app_fix_equals_opt}
    Z^\star
    =
    \operatorname{Fix}(\mathcal{H})
    :=
    \{z\in X\times Y:\mathcal{H}(z)=z\}.
\end{equation}
Furthermore,
\begin{equation}
\label{eq:app_reflection_fix}
    \operatorname{Fix}(\mathcal{R}_\alpha)
    =
    \operatorname{Fix}(\mathcal{H})
    =
    Z^\star.
\end{equation}
\end{lemma}

\begin{proof}[Proof]
Let
\begin{equation}
    z=\mathcal{H}(\hat z)=(x,y).
\end{equation}
\Eqref{eq:pdhg_iteration} gives
\begin{align}
    x&=
    \operatorname{proj}_X
    \left[
        \hat x-\frac{\eta}{\omega}(c-A^\top \hat y)
    \right],\label{eq:app_pdhg_x}\\
    y&=
    \operatorname{proj}_Y
    \left[
        \hat y+\omega\eta\left(b-A(2x-\hat x)\right)
    \right].\label{eq:app_pdhg_y}
\end{align}
Using the projection characterization
\begin{equation}
    v=\operatorname{proj}_C(a)
    \quad\Longleftrightarrow\quad
    a-v\in N_C(v),
\end{equation}
\Eqref{eq:app_pdhg_x} and \Eqref{eq:app_pdhg_y} imply
\begin{align}
    \frac{\omega}{\eta}(\hat x-x)+A^\top \hat y-c&\in N_X(x),\label{eq:app_proj_x}\\
    \frac{1}{\omega\eta}(\hat y-y)+b-A(2x-\hat x)&\in N_Y(y).\label{eq:app_proj_y}
\end{align}

Suppose first that $\hat z=\mathcal{H}(\hat z)$. Then $x=\hat x$ and $y=\hat y$, so
\begin{equation}
    A^\top \hat y-c\in N_X(\hat x),
    \qquad
    b-A\hat x\in N_Y(\hat y).
\end{equation}
Equivalently,
\begin{equation}
    0\in c-A^\top \hat y+N_X(\hat x),
    \qquad
    0\in A\hat x-b+N_Y(\hat y).
\end{equation}
By Lemma~\ref{lem:app_kkt}, $\hat z\in Z^\star$.

Conversely, let $\hat z\in Z^\star$. Then the KKT conditions imply
\begin{equation}
    A^\top \hat y-c\in N_X(\hat x),
    \qquad
    b-A\hat x\in N_Y(\hat y).
\end{equation}
The projection characterization gives
\begin{equation}
    \hat x=
    \operatorname{proj}_X
    \left[
        \hat x-\frac{\eta}{\omega}(c-A^\top \hat y)
    \right]
\end{equation}
and
\begin{equation}
    \hat y=
    \operatorname{proj}_Y
    \left[
        \hat y+\omega\eta(b-A\hat x)
    \right].
\end{equation}
Since the first equation gives $x=\hat x$, the second equation agrees with \Eqref{eq:app_pdhg_y}. Hence $\mathcal{H}(\hat z)=\hat z$.

Finally,
\begin{equation}
    \hat z-\mathcal{R}_\alpha(\hat z)
    =
    (1+\alpha)\bigl(\hat z-\mathcal{H}(\hat z)\bigr),
\end{equation}
and therefore $\operatorname{Fix}(\mathcal{R}_\alpha)=\operatorname{Fix}(\mathcal{H})$.
\end{proof}

\begin{lemma}[Metric and nonexpansiveness]
\label{lem:app-metric-nonexpansiveness}
For fixed $\omega,\eta>0$, define
\begin{equation}
\label{eq:pdhg-metric}
    Q\left(\omega,\eta\right):=
    \begin{pmatrix}
        \frac{\omega}{\eta}I_n&A^\top\\
        A&\frac{1}{\omega\eta}I_m
    \end{pmatrix}.
\end{equation}
Then $Q\succ0$ if and only if $\eta\lVert A\rVert_2<1$.
Under this condition, $\mathcal H=\mathcal H\left(\cdot;\omega,\eta\right)$ is firmly nonexpansive in the norm $\lVert v\rVert_Q:=\sqrt{v^\top Qv}$,
\begin{equation}
    \left\lVert\mathcal H\left(\hat z\right)-\mathcal H\left(\hat z'\right)\right\rVert_Q^2
    \leq\left(\mathcal H\left(\hat z\right)-\mathcal H\left(\hat z'\right)\right)^\top
    Q\left(\hat z-\hat z'\right).
\end{equation}
Consequently, it is nonexpansive in this norm, and the fixed-point residual
\begin{equation}
\label{eq:fpr}
    M^{\mathrm{fp}}\left(\hat z;\omega,\eta\right)
    :=\left\lVert\mathcal H\left(\hat z;\omega,\eta\right)-\hat z\right\rVert_Q
\end{equation}
vanishes if and only if $\hat z\in Z^\star$.
\end{lemma}

\begin{proof}[Proof]
The Schur complement of the primal diagonal block in \Eqref{eq:pdhg-metric} gives
\begin{equation}
    Q\succ0\quad\Longleftrightarrow\quad
    I_m-\eta^2AA^\top\succ0
    \quad\Longleftrightarrow\quad\eta\lVert A\rVert_2<1.
\end{equation}
Fix parameters satisfying this condition. For $z=\mathcal H\left(\hat z;\omega,\eta\right)$, adding the two projection inequalities gives, for every $z'\in X\times Y$,
\begin{equation}
    \begin{aligned}
        \left(\hat z-z\right)^\top Q\left(z'-z\right)
        &\leq\left(c-A^\top y\right)^\top\left(x'-x\right)\\
        &\quad+\left(Ax-b\right)^\top\left(y'-y\right).
    \end{aligned}
\end{equation}
Apply this inequality to outputs $z=\mathcal H\left(\hat z\right)$ and $z'=\mathcal H\left(\hat z'\right)$, interchange the primed and unprimed variables, and add. The coupling terms cancel, yielding
\begin{equation}
    \lVert z-z'\rVert_Q^2
    \leq\left(z-z'\right)^\top Q\left(\hat z-\hat z'\right).
\end{equation}
This proves firm nonexpansiveness~\citep{lu2024halpern}. Applying Cauchy--Schwarz in the $Q$-inner product gives $\lVert z-z'\rVert_Q\leq\lVert\hat z-\hat z'\rVert_Q$. Positive definiteness gives $M^{\mathrm{fp}}\left(\hat z;\omega,\eta\right)=0$ if and only if $\mathcal H\left(\hat z\right)=\hat z$. Lemma~\ref{lem:app_fixed_point} then gives $\hat z\in Z^\star$.
\end{proof}

\subsection{Proof of Proposition~\ref{prop:ga-pdhg-optimality}}
\label{app:ga-pdhg-optimality}

\begin{proof}[Proof of Proposition~\ref{prop:ga-pdhg-optimality}]
By convergence, write $\hat z_t\to\bar{\hat z}$ and $z_t\to\bar z$.
Then
\begin{equation}
\label{eq:ga-pdhg-vanishing-differences}
    \left\lVert z_t-z_{t-1}\right\rVert_2\longrightarrow0,
    \qquad\left\lVert\hat z_{t+1}-\hat z_t\right\rVert_2\longrightarrow0.
\end{equation}
Since $\boldsymbol\beta_t$ is bounded, choose $B_\beta<\infty$ such that $\lVert\boldsymbol\beta_t\rVert_2\leq B_\beta$ for all $t$.
The fixed, finite history length gives, for sufficiently large $t$,
\begin{equation}
\label{eq:ga-pdhg-vanishing-history}
    \begin{aligned}
        \left\lVert\mathbf K_t\boldsymbol\beta_t\right\rVert_2
        &\leq B_\beta\left(\sum_{j=0}^{k-1}
        \left\lVert z_{t-j}-z_{t-j-1}\right\rVert_2^2\right)^{1/2}
        \longrightarrow0.
    \end{aligned}
\end{equation}
For $k=0$, the history correction and the sum are zero.
At a non-restart step, rearranging \Eqref{eq:gen_acc_iter} gives
\begin{equation}
\label{eq:ga-pdhg-input-output-identity}
    \begin{pmatrix}
        \left(1+\alpha_t^x\right)\left(x_t-\hat x_t\right)\\
        \left(1+\alpha_t^y\right)\left(y_t-\hat y_t\right)
    \end{pmatrix}
    =\hat z_{t+1}-\hat z_t-\mathbf K_t\boldsymbol\beta_t
    -\zeta_t\left(z_\star-z_t\right).
\end{equation}
Set $C_\delta:=\max\{1,\delta^{-1}\}$. The lower bound $1+\alpha_t^x,1+\alpha_t^y\geq\delta$ on nonrestart steps, the triangle inequality, and the vanishing anchor correction assumed in Proposition~\ref{prop:ga-pdhg-optimality} give
\begin{equation}
    \left\lVert z_t-\hat z_t\right\rVert_2
    \leq C_\delta\left(\left\lVert\hat z_{t+1}-\hat z_t\right\rVert_2
    +\left\lVert\mathbf K_t\boldsymbol\beta_t\right\rVert_2
    +\left\lVert\zeta_t\left(z_\star-z_t\right)\right\rVert_2\right)
    \longrightarrow0.
\end{equation}
At a restart, $\hat z_{t+1}=z_t$, so the same bound holds because $C_\delta\geq1$.
Thus $\bar{\hat z}=\bar z$.
Since $\eta_t$ and $\omega_t$ are uniformly positive and bounded, there exist constants satisfying
\begin{equation}
\label{eq:ga-pdhg-parameter-bounds}
    0<\underline\eta\leq\eta_t\leq\overline\eta<\infty,
    \qquad 0<\underline\omega\leq\omega_t\leq\overline\omega<\infty.
\end{equation}
Hence a subsequence $t_j\to\infty$ satisfies
\begin{equation}
    \left(\omega_{t_j},\eta_{t_j}\right)
    \longrightarrow\left(\omega_\infty,\eta_\infty\right)
    \in\left[\underline\omega,\overline\omega\right]
    \times\left[\underline\eta,\overline\eta\right].
\end{equation}
Euclidean projections onto the nonempty closed convex sets $X$ and $Y$ are continuous, so $\mathcal H$ is jointly continuous in its input and positive parameters.
Taking limits in \Eqref{eq:base_PDHG} along this subsequence gives
\begin{equation}
    \bar z=\lim_{j\to\infty}z_{t_j+1}
    =\lim_{j\to\infty}\mathcal H\left(\hat z_{t_j+1};\omega_{t_j},\eta_{t_j}\right)
    =\mathcal H\left(\bar z;\omega_\infty,\eta_\infty\right).
\end{equation}
Thus $\bar z$ is a fixed point of $\mathcal H\left(\cdot;\omega_\infty,\eta_\infty\right)$.
By Lemma~\ref{lem:app_fixed_point}, $\bar z\in Z^\star$, completing the proof.
\end{proof}

\subsection{Satisfying the assumptions in implementation}
\label{app:implementation-assumptions}

We relate the deterministic inference rule in Appendix~\ref{app:controller-policy} to the assumptions of Proposition~\ref{prop:ga-pdhg-optimality}.
The trained network parameters are fixed during each solve.
The implementation enforces the coefficient bounds and makes the anchor correction vanish whenever the output sequence converges.
Convergence of both $\{\hat z_t\}$ and $\{z_t\}$ remains a premise of the proposition.

\paragraph{Bounded history coefficients and positive relaxation.}
Each continuous actor has a final tanh hidden layer followed by a linear output layer.
For any scalar mean output, write
\begin{equation}
\label{eq:ga-pdhg-bounded-policy-mean}
    \mu\left(o\right)=w_\mu^\top h_\mu\left(o\right)+b_\mu,
    \qquad h_\mu\left(o\right)\in[-1,1]^{64},
    \qquad |\mu\left(o\right)|\leq\lVert w_\mu\rVert_1+|b_\mu|.
\end{equation}
The right-hand side is a finite constant for the trained network, independent of the observation and iteration.
At inference, $\boldsymbol\beta_t=\boldsymbol\mu_t^\beta$, so the fixed number $k$ of history coefficients is uniformly bounded.
The exponential transforms give $\alpha_t^x,\alpha_t^y>0$, satisfying the proposition's relaxation lower bound with $\delta=1$, and a bounded positive anchor multiplier $\widetilde\zeta_t=\exp\left(\mu_t^\zeta\right)\leq C_\zeta$ for some $C_\zeta<\infty$.

\paragraph{Step-size and primal-weight bounds.}
Let $a_A>0$ denote the fixed spectral-norm denominator, including the numerical safeguard, for the matrix supplied to PDHG.
Let $L_\omega>0$ be the configured primal-weight range in \Eqref{eq:learned_omega}.
The sigmoid and tanh transforms give
\begin{equation}
    \frac{0.1}{a_A}\leq\eta_t\leq\frac{5}{a_A},
    \qquad\omega_0e^{-L_\omega}\leq\omega_t\leq\omega_0e^{L_\omega}.
\end{equation}
Since $a_A$ and $\omega_0$ are positive and finite for each LP, both parameters are uniformly positive and bounded.
The weight is set relative to $\omega_0$ at every iteration, so these bounds do not accumulate over successive updates.
Proposition~\ref{prop:ga-pdhg-optimality} does not require $\eta_t\lVert A\rVert_2<1$.

\paragraph{Vanishing normalized anchor correction.}
The implemented coefficient is
\begin{equation}
\label{eq:ga-pdhg-anchor-parameterization}
    \zeta_t=
    \begin{cases}
        \displaystyle\widetilde\zeta_t
        \frac{\left\lVert z_t-z_{t-1}\right\rVert_2}
        {\left\lVert z_\star-z_t\right\rVert_2},&z_\star\neq z_t,\\[6pt]
        0,&z_\star=z_t.
    \end{cases}
\end{equation}
The zero case is handled explicitly before division.
For either case, if $z_t\to\bar z$, then
\begin{equation}
    \left\lVert\zeta_t\left(z_\star-z_t\right)\right\rVert_2
    \leq C_\zeta\left\lVert z_t-z_{t-1}\right\rVert_2\longrightarrow0.
\end{equation}
Thus the anchor correction vanishes as required by Proposition~\ref{prop:ga-pdhg-optimality}, independently of the restart pattern.
Only the anchor correction must vanish; the effective coefficient $\zeta_t$ need not itself be bounded.

\paragraph{Alternative Halpern-style parameterization.}

Using the epoch counter $s$ from Section~\ref{sec:generalized-acceleration}, an alternative parameterization is
\begin{equation}
\label{eq:ga-pdhg-halpern-parameterization}
    \zeta_t=\min\left\{1,\frac{\widetilde\zeta_t}{s+1}\right\},
    \qquad0\leq\zeta_t\leq\min\left\{1,\frac{C_\zeta}{s+1}\right\}.
\end{equation}
The associated relaxation coefficients
\begin{equation}
    \alpha_t^x=\exp\left(u_t^{\alpha^x}\right)\left(1-\zeta_t\right),
    \qquad\alpha_t^y=\exp\left(u_t^{\alpha^y}\right)\left(1-\zeta_t\right)
\end{equation}
are nonnegative, satisfying the proposition's relaxation lower bound with $\delta=1$.
Suppose $z_t\to\bar z$.
If only finitely many restarts occur, then $z_\star$ is eventually fixed and $s\to\infty$, so
\begin{equation}
    \left\lVert\zeta_t\left(z_\star-z_t\right)\right\rVert_2
    \leq\frac{C_\zeta}{s+1}\left\lVert z_\star-z_t\right\rVert_2
    \longrightarrow0.
\end{equation}
If infinitely many restarts occur, the iteration at which the current anchor was set tends to infinity.
Since each new anchor is an output iterate, $z_\star\to\bar z$.
Using $0\leq\zeta_t\leq1$ gives
\begin{equation}
    \left\lVert\zeta_t\left(z_\star-z_t\right)\right\rVert_2
    \leq\left\lVert z_\star-z_t\right\rVert_2\longrightarrow0.
\end{equation}
Thus the anchor correction vanishes as required by Proposition~\ref{prop:ga-pdhg-optimality} for either restart pattern.

During training, Gaussian exploration has unbounded support, so the uniform bounds on $\boldsymbol\beta_t$ and $\widetilde\zeta_t$ established here do not extend automatically to an infinite sampled trajectory.
Proposition~\ref{prop:ga-pdhg-optimality} applies to any trajectory satisfying its hypotheses, whether its controls are deterministic or random.

\subsection{Coefficient mappings for classical acceleration}
\label{app:ga-pdhg-specializations}

Fix $\omega,\eta>0$ and write $\mathcal H=\mathcal H\left(\cdot;\omega,\eta\right)$.
Within an uninterrupted run, write $u_s:=\hat z_t$ and $v_s:=z_t=\mathcal H\left(u_s\right)$ at local iteration $s$.
The following mappings concern the acceleration recurrence with each method's standard initialization, independently of the rollout initialization and restart branch in Algorithm~\ref{alg:gen-pdhg-rollout}.

\paragraph{Halpern and reflected Halpern.}
Set $u_0=z_\star$ and $\boldsymbol\beta_t=0$.
Substituting the first two rows of Table~\ref{tab:ga-pdhg-specializations} into \Eqref{eq:gen_acc_iter} gives, respectively,
\begin{equation}
\begin{aligned}
u_{s+1}&=\frac{s+1}{s+2}\mathcal H\left(u_s\right)+\frac{1}{s+2}u_0,\\
u_{s+1}&=\frac{s+1}{s+2}\left(2\mathcal H\left(u_s\right)-u_s\right)+\frac{1}{s+2}u_0.
\end{aligned}
\end{equation}
These are Halpern iteration on $\mathcal H$ and on its reflection $2\mathcal H-I$, respectively~\citep{halpern1967fixed,lu2024halpern}.

\paragraph{Krasnosel'ski\u\i--Mann.}
For constant $0<\lambda<2$, setting $\alpha_t^x=\alpha_t^y=\lambda-1$, $\boldsymbol\beta_t=0$, and $\zeta_t=0$ gives
\begin{equation}
u_{s+1}=\left(1-\lambda\right)u_s+\lambda\mathcal H\left(u_s\right).
\end{equation}
This is the KM relaxation of $\mathcal H$, with overrelaxation when $\lambda>1$~\citep{bot2023fastkm}.

\paragraph{Fast KM.}
For $\nu>2$, $k\geq1$, and $s\geq1$, the final row of Table~\ref{tab:ga-pdhg-specializations} gives
\begin{equation}
\begin{aligned}
u_{s+1}
&=v_s+\left(\frac{\nu}{2\left(s+\nu\right)}-1\right)\left(v_s-u_s\right)
  +\frac{s}{s+\nu}\left(v_s-v_{s-1}\right)\\
&=\left(1-\frac{\nu}{2\left(s+\nu\right)}\right)u_s
  +\frac{\nu}{2\left(s+\nu\right)}\mathcal H\left(u_s\right)
  +\frac{s}{s+\nu}\left(\mathcal H\left(u_s\right)-\mathcal H\left(u_{s-1}\right)\right).
\end{aligned}
\end{equation}
This is equation~(2.8) of \citet{bot2023fastkm}, with their damping parameter $\alpha$ replaced by $\nu$ and their step parameter set to one.
Initialize $u_0,u_1$ and store $v_0=\mathcal H\left(u_0\right)$ and $v_1=\mathcal H\left(u_1\right)$.
The general Fast KM recurrence with a nonunit step also contains an input-difference term, so the table specifies the unit-step variant.

\paragraph{Anderson coefficients.}
The vector $\boldsymbol\beta_t^{\mathrm{AA}}\in\mathbb R^k$ specifies how much of each past output difference to add to the current PDHG output $z_t$.
Anderson acceleration chooses these coefficients by using past changes in the fixed-point residual to approximately cancel the current residual.
For the PDHG input--output pairs $\left(\hat z_i,z_i\right)$, where $z_i=\mathcal H\left(\hat z_i\right)$, define
\begin{equation}
d_t:=z_t-\hat z_t,\qquad
\mathbf F_t:=\begin{bmatrix}d_t-d_{t-1}&\cdots&d_{t-k+1}-d_{t-k}\end{bmatrix}.
\end{equation}
The residual $d_t$ vanishes exactly when $\hat z_t$ is a fixed point of $\mathcal H$.
The columns of $\mathbf F_t$ record residual changes over the same history as the output differences in $\mathbf K_t$.
For trial coefficients $b\in\mathbb R^k$, the residual model $d_t+\mathbf F_t b$ combines the current residual with these past changes.
Undamped type-II Anderson acceleration chooses $b$ so that $\mathbf F_t b$ cancels $d_t$ as closely as possible in the weighted norm $\lVert v\rVert_Q^2=v^\top Qv$.
Here, $Q\succ0$ is the metric in \Eqref{eq:pdhg-metric}, with $\eta\lVert A\rVert_2<1$.
Using the available history columns and fixing all unavailable coefficients to zero, this gives
\begin{equation}
\label{eq:ga-pdhg-aa-coefficients}
\boldsymbol\beta_t^{\mathrm{AA}}\in
\operatorname*{arg\,min}_{b\in\mathbb R^k}\lVert d_t+\mathbf F_t b\rVert_Q^2.
\end{equation}
The optimization variables are the history coefficients, and the objective is the squared norm of this residual model, not the exact residual at the next iterate.
Apply the fitted coefficient of each residual difference to the matching output difference to form the next PDHG input,
\begin{equation}
\label{eq:ga-pdhg-aa-update}
\hat z_{t+1}=z_t+\mathbf K_t\boldsymbol\beta_t^{\mathrm{AA}}.
\end{equation}
This is the undamped Anderson update in difference form~\citep{anderson1965iterative}, recovered from \Eqref{eq:gen_acc_iter} by setting $\alpha_t^x=\alpha_t^y=\zeta_t=0$.
For a full window and $b=\boldsymbol\beta_t^{\mathrm{AA}}$, both the residual model and the update use the affine weights $1+b_1$, $b_{j+1}-b_j$ for $1\leq j<k$, and $-b_k$ on the current and past residuals or outputs, respectively.
These weights sum to one.

Table~\ref{tab:ga-pdhg-specializations} describes the algebraic GA-PDHG family.
The exponential parameterization in Appendix~\ref{app:controller-policy} restricts the learned extrapolation coefficients to positive values, so it does not implement the negative coefficients of underrelaxed KM or the displayed Fast KM specialization.

\section{Details on PDHG-based LP solver enhancements}
\label{app:pdlp-enhancements}

This appendix gives the enhancement rules summarized in Section~\ref{sec:pdlp-acceleration}.
Table~\ref{tab:pdlp-variants} compares these enhancements across representative solvers.
The four subsections follow the table columns.

\begin{table}[!htbp]
\centering
\small
\setlength{\tabcolsep}{3pt}
\renewcommand{\arraystretch}{1.15}
\caption{Representative PDHG-based LP solver papers and their enhancement techniques.}
\label{tab:pdlp-variants}
\begin{tabularx}{\linewidth}{@{}>{\raggedright\arraybackslash}p{0.34\linewidth}>{\centering\arraybackslash}p{0.10\linewidth}>{\centering\arraybackslash}p{0.11\linewidth}>{\centering\arraybackslash}p{0.17\linewidth}>{\raggedright\arraybackslash}X@{}}
\toprule
Paper / solver & Step size & Primal weight & Restarting criteria & Acceleration method\\
\midrule
PDLP~\citep{applegate2021practical}
    & Adaptive & EMA & Normalized\newline duality gap & Averaging\\
raPDHG~\citep{applegate2023faster}
    & Constant & Constant & Normalized\newline duality gap & Averaging\\
cuPDLP-C~\citep{lu2023cupdlpc} & Adaptive & EMA & KKT error & Averaging\\
cuPDLP.jl~\citep{lu2023cupdlp} & Adaptive & EMA & KKT error & Averaging\\
rHPDHG / r$^2$HPDHG\newline \citep{lu2024halpern}
    & Adaptive & EMA & Fixed-point\newline residual & Halpern /\newline Reflected Halpern\\
MPAX (LP)~\citep{lu2024mpax}
    & Adaptive & EMA & Fixed-point\newline residual & Reflected Halpern\\
cuPDLPx~\citep{lu2025cupdlpx}
    & Constant & PID & Fixed-point\newline residual & Reflected Halpern\\
AA-PDHG \citep{zhou2025anderson}
    & Constant & $^\dagger$EMA & N/A & Anderson \\
\bottomrule
\end{tabularx}
\par\smallskip
\begin{minipage}{\linewidth}
\footnotesize
EMA denotes exponential averaging in log space.
PID denotes PID controller.
N/A denotes no epoch restart.
Paired acceleration entries follow the order of the variants.
$^\dagger$AA-PDHG uses periodic EMA and resets the Anderson history at weight updates.
\end{minipage}
\end{table}

\subsection{Step size}
\label{app:pdlp-step-size}
raPDHG, cuPDLPx, and AA-PDHG use a constant step size $\eta_t\equiv\eta$.
The condition $\eta\lVert A\rVert_2<1$ ensures nonexpansiveness of the PDHG map in the induced norm (Appendix~\ref{app:pdhg-fixed-point}).
Choosing a constant step size from this condition can be conservative~\citep{applegate2021practical}, because it must hold for all possible updates, even when the current update permits a larger step.
PDLP, cuPDLP.jl/C, MPAX, and rHPDHG/r$^2$HPDHG instead choose the step size $\eta_t$ adaptively from trial updates.
Concretely, for a trial update $z_{t+1}=\mathcal H\left(\hat z_{t+1};\omega_t,\eta_t\right)$, the step size acceptance test verifies
\begin{equation}
\label{eq:pdlp-stepsize-test}
    \eta_t\leq
    \frac{\lVert z_{t+1} - \hat z_{t+1}\rVert_{\omega_t}^2}
    {2\left|\left(y_{t+1}-\hat y_{t+1}\right)^\top A\left(x_{t+1}-\hat x_{t+1}\right)\right|}, \quad \text{where} \quad \lVert z\rVert_\omega^2:=\omega\lVert x\rVert_2^2+\omega^{-1}\lVert y\rVert_2^2.
\end{equation}
If the test fails, it reduces $\eta_t$ and recomputes $z_{t+1}$ until the condition holds.
Checking only the current trial can permit steps beyond the sufficient spectral-norm bound.
Rejected trials incur extra matrix--vector products.
Although the test is motivated by the nonexpansiveness proof of the PDHG map, to the best of our knowledge, no convergence guarantee has been established for the adaptive step-size rule used in practical PDLP-style solvers~\citep{applegate2021practical,lu2024halpern}.

\subsection{Primal weight}
\label{app:pdlp-primal-weight}
The primal weight adjusts the relative primal and dual step sizes to balance the scales of the primal and dual variables.
The convergence of standard PDHG is proven for a constant $\omega_t\equiv\omega>0$ and a constant step size $\eta<1/\lVert A\rVert_2$~\citep{chambolle2011first,applegate2023faster}.
Adaptive primal-weight updates are not covered by this fixed-weight convergence guarantee.
In the hand-designed adaptive schemes reviewed here, the primal weight is updated at the end of each block of iterations, with block boundaries set by restarts or fixed update intervals. It remains constant within each block.
Superscript $j$ indexes blocks, subscript $t$ indexes individual iterations, and $\omega^j>0$ is the weight in block $j$.
Let $\left(x^j,y^j\right)$ denote the starting point of block $j$. The displacement norms between successive block starting points are
\begin{equation}
\label{eq:pdlp-weight-displacements}
    d_x^j:=\left\lVert x^{j+1}-x^j\right\rVert_2,
    \qquad d_y^j:=\left\lVert y^{j+1}-y^j\right\rVert_2.
\end{equation}
For $d_x^j,d_y^j$, define the log-balance error
\begin{equation}
\label{eq:pdlp-weight-balance}
    b^j:=\log\left(\frac{\omega^j d^j_{x}}{d^j_{y}}\right).
\end{equation}
PDLP, cuPDLP-C/jl, rHPDHG/r$^2$HPDHG, MPAX, and AA-PDHG use exponential moving averaging (EMA) in log space, whereas cuPDLPx uses proportional--integral--derivative (PID) control.
At the end of block $j$, the weight for the next block is computed as
\begin{equation}
\label{eq:pdlp-weight-background}
    \log\omega^{j+1}=\log\omega^j-
    \begin{cases}
        \theta_\omega b^j,&\text{EMA},\\
        K_p b^j+K_i\displaystyle\sum_{\ell=0}^{j}b^\ell+K_d\left(b^j-b^{j-1}\right),&\text{PID},
    \end{cases}
\end{equation}
where $\theta_\omega\in[0,1]$ is the smoothing coefficient, and $K_p,K_i,K_d$ are PID gains.
Both rules implement feedback control to balance the weighted displacements $\sqrt{\omega^j}d_x^j$ and $d_y^j/\sqrt{\omega^j}$, corresponding to $b^j=0$.
EMA changes $\log\omega^j$ by $-\theta_\omega b^j$, so the correction is proportional to the current error.
PID control uses past errors through its integral and derivative terms~\citep{lu2025cupdlpx}.

\subsection{Restarting}
\label{app:pdlp-restarting}
Restarting reinitializes PDHG from a selected point $z_\star$ and resets its past averaging or acceleration state.
When restarting is combined with averaging or Halpern acceleration, suitable parameter choices and restart rules exploit LP sharpness to achieve accelerated linear convergence~\citep{applegate2023faster,lu2024halpern}.
Restart timing affects both the theoretical iteration complexity and the practical convergence of restarted PDHG.
A fixed-frequency scheme restarts after a prescribed number of iterations.
In practice, the restarted solvers in Table~\ref{tab:pdlp-variants} use adaptive restart heuristics based on the conditions
\begin{subequations}
\label{eq:pdlp_restart_condition}
\begin{align}
    & M_t \leq \beta^{\mathrm{suf}} M_\star
        && \text{(sufficient progress)}, \label{eq:sufficient_progress} \\
    & M_t \leq \beta^{\mathrm{nec}} M_\star
      \quad\text{and}\quad M_t > M_{t-1}
        && \text{(loss of progress)}, \label{eq:loss_of_progress} \\
    & s+1 \geq \beta^{\mathrm{art}}t
        && \text{(long iterations after restart)}, \label{eq:long_epoch}
\end{align}
\end{subequations}
Here, $M_t := M(z_t)$ is a restart metric evaluated at iterate $z_t$, and $M_\star := M(z_\star)$ is its value at the restart anchor $z_\star$.  In \Eqref{eq:long_epoch}, $s$ is the zero-based index within the current epoch, so $s+1$ counts its completed PDHG steps, while $t\geq1$ counts all completed PDHG steps.
These conditions are checked at a prescribed frequency using thresholds satisfying $0 < \beta^{\mathrm{suf}} < \beta^{\mathrm{nec}} < 1$ and $\beta^{\mathrm{art}} \in \left(0,1\right)$.
The metric $M$, manually tuned constant thresholds ($\beta^{\text{suf}}$, $\beta^{\text{nec}}$, $\beta^{\text{art}}$), and the condition-checking intervals differ across solvers, as do initialization rules such as more frequent checks during the first few iterations.

\subsection{Acceleration}
\label{app:pdlp-acceleration}
An acceleration step constructs the next PDHG input or a restart candidate from the current PDHG update and past iterates.
The solvers in Table~\ref{tab:pdlp-variants} use \textit{Averaging}, \textit{Halpern}, or \textit{Anderson} accelerations.

\paragraph{Averaging.}
Averaging uses the step-size weighted average
\[
    \bar z_t=\frac{\sum_{i=t-s}^{t}\eta_{i-1}z_i}{\sum_{i=t-s}^{t}\eta_{i-1}}
\]
as the restart candidate.
When primal--dual iterates oscillate around a solution, averaging can cancel deviations in opposite directions.
Restarted averaging with a fixed step size $\eta$ and primal weight $\omega$ exploits LP sharpness to achieve accelerated linear convergence under the prescribed restart rules based on the normalized duality gap or KKT error~\citep{applegate2023faster,lu2023cupdlp}.

\paragraph{Halpern acceleration.}
Halpern acceleration combines the current PDHG input $\hat z_t$, output $z_t$, and fixed restart anchor $z_\star$ to form the next PDHG input
\begin{equation}
\label{eq:pdlp-halpern-background}
    \hat z_{t+1} = z_t + \frac{\alpha (s+1)}{s+2} \left( z_t-\hat z_t\right) +\frac{1}{s+2}\left(z_\star - z_t\right).
\end{equation}
Here, $\alpha$ is the reflection coefficient, with $\alpha=0$ recovering standard Halpern acceleration and $\alpha=1$ yielding reflected Halpern acceleration~\citep{lu2024halpern}.
The anchor coefficient $1/\left(s+2\right)$ decreases as $s$ increases and returns to $1/2$ at each restart, when $s$ is reset to zero.
For feasible and bounded LPs, both variants achieve an $O\left(1/\left(s+1\right)\right)$ fixed-point residual bound with fixed $\eta,\omega>0$ satisfying $\eta\lVert A\rVert_2<1$.
When combined with the restart schemes analyzed by \citet{lu2024halpern}, they achieve accelerated linear convergence.

\paragraph{Anderson acceleration.}
Anderson acceleration corrects the current PDHG output using recent output differences~\citep{zhou2025anderson}.
Using the $k$ most recent PDHG output differences, Anderson acceleration takes the form
\begin{equation}
\label{eq:pdlp-anderson-background}
    \hat z_{t+1}=z_t+\mathbf K_t\boldsymbol\beta_t^{\mathrm{AA}},\qquad
    \mathbf K_t:=\begin{bmatrix}z_t-z_{t-1}&\cdots&z_{t-k+1}-z_{t-k}\end{bmatrix}.
\end{equation}
The coefficients $\boldsymbol\beta_t^{\mathrm{AA}}\in\mathbb R^k$ are computed from recent PDHG input--output pairs $(\hat z_i,z_i)$.
A small least-squares problem selects coefficients that combine past residual changes to approximately cancel the current residual $z_t-\hat z_t$.
The same coefficients are then applied to the output differences to construct the next PDHG input.
The resulting history correction has a quasi-Newton interpretation and can reduce error components that decay slowly under basic PDHG.

These acceleration methods form linear combinations of the current and past iterates.
Averaging uses step-size weights, Halpern acceleration uses coefficients prescribed by convergence analysis, and Anderson acceleration computes coefficients by solving a least-squares problem.

\section{Further Details on Problem Classes}
\label{app:lp-problem-classes}

We describe the transportation, auction, minimum-cost flow, robust inventory management, and $L_\infty$ polynomial regression generators used in Section~\ref{sec:experimental-setup}.
The transportation, auction, and minimum-cost flow generators construct primal--dual solutions satisfying the KKT conditions, so their optimal values are known without an auxiliary solve.
The inventory generator solves each generated LP, while the regression generator also refines the resulting active set, as detailed below.
Nonzero counts refer to the stacked constraint matrix $A$ defined after \Eqref{eq:LP}, with variable bounds stored separately in $X$.
Here $\operatorname{nnz}(A)$ counts the nonzero entries, and matrix density is $\operatorname{nnz}\left(A\right)/(mn)$, the complement of the zero fraction. Symbols for problem-specific data in the following subsections are local to each LP class.

\subsection{Transport}
\label{app:transport}

Balanced transportation allocates mass from $s$ sources to $t$ destinations at minimum cost. In this subsection only, $s,t$ are problem dimensions.
Let $a\in\mathbb{R}_+^s$ and $d\in\mathbb{R}_+^t$ be the supply and demand vectors with $\sum_{i=1}^s a_i=\sum_{j=1}^t d_j$, and let $c_{ij}\in\mathbb{R}$ be the unit cost from source $i$ to destination $j$.
For transported amounts $x_{ij}\geq0$, the LP is
\begin{equation}
\label{eq:transport-lp}
\begin{aligned}
    \min_{x_{ij}\geq0}\quad
    &\sum_{i=1}^s\sum_{j=1}^t c_{ij}x_{ij}\\
    \text{subject to}\quad
    &\sum_{j=1}^t x_{ij}=a_i, \qquad i=1,\ldots,s,\\
    &\sum_{i=1}^s x_{ij}=d_j, \qquad j=1,\ldots,t.
\end{aligned}
\end{equation}
The generator obtains the marginals from a feasible plan supported on a spanning tree of the complete bipartite graph, optionally with additional support arcs.
After vectorization, the LP has $n=st$ variables, $m_e=m=s+t$ equality constraints, and no inequality rows.
Each column of $A=P$ contains two unit entries, one for its source and one for its destination, giving
\begin{equation}
    \operatorname{nnz}\left(A\right)=2st=2n,
    \qquad \frac{\operatorname{nnz}\left(A\right)}{mn}=\frac{2}{s+t}.
\end{equation}
All $st$ transport arcs remain in the LP, so changing the optimal plan's support or introducing cost ties does not change matrix sparsity.
One mass-balance equality is redundant and is retained.

\subsection{Auction}
\label{app:auction}

We use a capacitated LP relaxation of combinatorial auction winner determination, where bids request bundles of items.
For $m$ items and $n$ bids, let $B\in\{0,1\}^{m\times n}$ have $B_{ij}=1$ when bid $j$ includes item $i$, let $v\in\mathbb{R}^n$ contain bid values, and let $d\in\mathbb{R}_+^m$ contain item capacities.
The variable $x_j\in[0,1]$ is the accepted fraction of bid $j$, and the LP is
\begin{equation}
\label{eq:auction-lp}
    \max_{x\in\mathbb{R}^n}\quad v^\top x
    \quad\text{subject to}\quad Bx\leq d,\qquad 0\leq x\leq\mathbf{1}.
\end{equation}
In \Eqref{eq:LP}, this corresponds to $c=-v$, $G=-B$, and $h=-d$, with $m_e=0$ and $m_i=m$.
The generated capacities are not restricted to one.
Each bundle contains $\kappa\in\{1,\ldots,m\}$ distinct items sampled without replacement, so
\begin{equation}
    \operatorname{nnz}\left(A\right)=\kappa n,
    \qquad \frac{\operatorname{nnz}\left(A\right)}{mn}=\frac{\kappa}{m}.
\end{equation}
Bundle size determines column sparsity, while the item-popularity parameter changes row degrees and overlap among bids without changing the nonzero count.
The finite bounds allow both full and partial bid acceptance in the relaxation.

\subsection{Min-cost Flow}
\label{app:min-cost-flow}

Capacitated minimum-cost flow routes a commodity through a directed network while satisfying node balances and arc capacities.
Let $V$ and $E$ be the node and arc sets, with $m=|V|$ and $n=|E|$, and define the incidence matrix $B\in\mathbb{R}^{m\times n}$ to have $-1$ at the tail and $+1$ at the head of each arc.
For arc costs $c\in\mathbb{R}^n$, capacities $u\in\mathbb{R}_{>0}^n$, and net inflows $d\in\mathbb{R}^m$ satisfying $\mathbf{1}^\top d=0$, the flow vector $x\in\mathbb{R}^n$ solves
\begin{equation}
\label{eq:min-cost-flow-lp}
    \min_{x\in\mathbb{R}^n}\quad c^\top x
    \quad\text{subject to}\quad Bx=d,\qquad 0\leq x\leq u.
\end{equation}
The generator randomly orients a spanning tree and adds distinct directed arcs without self-loops, ensuring that the underlying undirected graph is connected.
Here $A=P=B$, $m_e=m$, and $m_i=0$, and each column contains exactly two nonzeros, giving
\begin{equation}
    \operatorname{nnz}\left(A\right)=2n,
    \qquad \frac{\operatorname{nnz}\left(A\right)}{mn}=\frac{2}{m}.
\end{equation}
Adding arcs increases row degrees and the total nonzero count, while column sparsity remains fixed.
The incidence matrix has rank $m-1$, and the redundant conservation row is retained.
Capacities affect which box constraints are active without changing matrix sparsity.

\subsection{Robust Inventory Management}
\label{app:robust-inventory}

We consider $N$ items over $T$ periods with uncertain demand and shared purchasing and storage limits.
For item $i=1,\ldots,N$ and period $t=1,\ldots,T$, let $x_t^{(i)}$, $u_t^{(i)}$, and $b_t^{(i)}$ denote inventory, purchases, and a nonnegative shortage buffer.
The initial inventory $x_1^{(i)}$ is given.
The nonnegative parameters $h_t^{(i)}$, $c_t^{(i)}$, $p_t^{(i)}$, $s_t^{(i)}$, and $v^{(i)}$ are the holding cost, purchase cost, selling price, shortage penalty, and unit volume.
Let $C$ and $V$ be the per-period capital budget and storage capacity.
Demand lies in a box, and purchases follow an itemwise causal affine policy,
\begin{equation}
\label{eq:inventory-policy-box}
    \mathcal D=\prod_{i=1}^N\prod_{t=1}^T
    \left[\underline d_t^{(i)},\bar d_t^{(i)}\right],
    \qquad
    u_t^{(i)}=\hat u_t^{(i)}+\sum_{\tau<t}k_{t,\tau}^{(i)}d_\tau^{(i)}.
\end{equation}
The nominal purchases $\hat u_t^{(i)}$, feedback gains $k_{t,\tau}^{(i)}$, and buffers $b_t^{(i)}$ are chosen before demand is observed.
Only purchases adapt to past demand, with no feedback between items.
The underlying robust problem is
\begin{equation}
\label{eq:inventory-robust}
\begin{aligned}
    \min_{\hat u,k,\,b\geq0}\ \max_{d\in\mathcal D}\quad
    &\sum_{i=1}^N\sum_{t=1}^T
    \left[h_t^{(i)}x_t^{(i)}+c_t^{(i)}u_t^{(i)}
    -p_t^{(i)}d_t^{(i)}+\left(p_t^{(i)}+s_t^{(i)}\right)b_t^{(i)}\right]\\
    \text{subject to}\quad
    &x_{t+1}^{(i)}=x_t^{(i)}+u_t^{(i)}-d_t^{(i)}+b_t^{(i)},
    \qquad i=1,\ldots,N,\quad t=1,\ldots,T,\\
    &x_t^{(i)}\geq0,\quad u_t^{(i)}\geq0,\quad
    \sum_i v^{(i)}x_t^{(i)}\leq V,\quad
    \sum_i c_t^{(i)}u_t^{(i)}\leq C,
    \qquad \forall d\in\mathcal D.
\end{aligned}
\end{equation}
The last line applies for $t=1,\ldots,T$ and, for the itemwise inequalities, every $i$.
Holding costs are charged on $x_t$, the inventory entering period $t$.
The generated problem has no terminal constraint on $x_{T+1}$.
The buffer $b_t$ is a deterministic feasibility slack, rather than the realized positive-part shortage, and has no demand-dependent upper bound.

\paragraph{LP formulation.}
Define the demand-independent inventory $a_t^{(i)}$ and demand coefficients $\beta_{t,\tau}^{(i)}$ by
\begin{equation}
\label{eq:inventory-state-coefficients}
\begin{aligned}
    a_t^{(i)}&=x_1^{(i)}+\sum_{r<t}\left(\hat u_r^{(i)}+b_r^{(i)}\right),
    &\beta_{t,\tau}^{(i)}&=\sum_{r=\tau+1}^{t-1}k_{r,\tau}^{(i)}-1,\quad \tau<t,\\
    x_t^{(i)}&=a_t^{(i)}+\sum_{\tau<t}\beta_{t,\tau}^{(i)}d_\tau^{(i)}.
\end{aligned}
\end{equation}
Empty sums are zero.
The total cost is $\Gamma_0+\sum_{i,\tau}\gamma_\tau^{(i)}d_\tau^{(i)}$, where
\begin{equation}
\label{eq:inventory-cost-coefficients}
\begin{aligned}
    \Gamma_0&=\sum_{i=1}^N\sum_{t=1}^T
    \left[h_t^{(i)}a_t^{(i)}+c_t^{(i)}\hat u_t^{(i)}
    +\left(p_t^{(i)}+s_t^{(i)}\right)b_t^{(i)}\right],\\
    \gamma_\tau^{(i)}&=\sum_{t=\tau+1}^T
    \left[h_t^{(i)}\beta_{t,\tau}^{(i)}+c_t^{(i)}k_{t,\tau}^{(i)}\right]-p_\tau^{(i)}.
\end{aligned}
\end{equation}
For each coefficient $q\in\{k_{t,\tau}^{(i)},\beta_{t,\tau}^{(i)},\gamma_\tau^{(i)}\}$, introduce $q^+,q^-\geq0$ with $q=q^+-q^-$.
For complementary splits satisfying $q^+q^-=0$, the minimum and maximum of $qd_\tau^{(i)}$ over its interval are, respectively,
\begin{equation}
\label{eq:inventory-box-extrema}
    L_\tau^{(i)}\left(q^\pm\right)
    =\underline d_\tau^{(i)}q^+-\bar d_\tau^{(i)}q^- ,
    \qquad
    U_\tau^{(i)}\left(q^\pm\right)
    =\bar d_\tau^{(i)}q^+-\underline d_\tau^{(i)}q^- .
\end{equation}
Let $\mathcal R$ be the sum of all split variables,
\begin{equation}
\label{eq:inventory-split-penalty}
    \mathcal R=\sum_{i=1}^N
    \left[\sum_{\tau<t}\left(k_{t,\tau}^{(i),+}+k_{t,\tau}^{(i),-}
    +\beta_{t,\tau}^{(i),+}+\beta_{t,\tau}^{(i),-}\right)
    +\sum_{\tau=1}^T\left(\gamma_\tau^{(i),+}+\gamma_\tau^{(i),-}\right)\right].
\end{equation}
With a free epigraph variable $\eta$, the generated LP is
\begin{equation}
\label{eq:inventory-lp}
\begin{aligned}
    \min_{\eta,\hat u,b,k^\pm,\beta^\pm,\gamma^\pm}\quad
    &\eta+\varepsilon\mathcal R\\
    \text{subject to}\quad
    &\Gamma_0+\sum_{i=1}^N\sum_{\tau=1}^T
    U_\tau^{(i)}\left(\gamma_\tau^{(i),\pm}\right)\leq\eta,\\
    &\sum_{i=1}^Nc_t^{(i)}\left[\hat u_t^{(i)}+
    \sum_{\tau<t}U_\tau^{(i)}\left(k_{t,\tau}^{(i),\pm}\right)\right]\leq C,
    &&t=1,\ldots,T,\\
    &\sum_{i=1}^Nv^{(i)}\left[a_t^{(i)}+
    \sum_{\tau<t}U_\tau^{(i)}\left(\beta_{t,\tau}^{(i),\pm}\right)\right]\leq V,
    &&t=2,\ldots,T,\\
    &\hat u_t^{(i)}+\sum_{\tau<t}L_\tau^{(i)}\left(k_{t,\tau}^{(i),\pm}\right)\geq0,
    &&i=1,\ldots,N,\quad t=1,\ldots,T,\\
    &a_t^{(i)}+\sum_{\tau<t}L_\tau^{(i)}\left(\beta_{t,\tau}^{(i),\pm}\right)\geq0,
    &&i=1,\ldots,N,\quad t=2,\ldots,T,\\
    &b,\ k^\pm,\ \beta^\pm,\ \gamma^\pm\geq0.
\end{aligned}
\end{equation}
Here $a_t$, $\Gamma_0$, $k$, $\beta$, and $\gamma$ are substituted using their definitions above.
The definitions of $\beta$ and $\gamma$ impose linear equalities on the split variables.
The nominal purchases $\hat u$ are unrestricted.
Removing the common part of any split pair preserves these equalities and weakens every inequality, so complementary splits suffice without explicit complementarity constraints.
Thus $\varepsilon=0$ gives an exact LP reformulation of \Eqref{eq:inventory-robust}.
The benchmark uses $\varepsilon=10^{-6}$ and includes this finite regularization in its objective.

\paragraph{Generated instances.}
We use $T=10$, $x_1^{(i)}=0$, $v^{(i)}=1$, $c_t^{(i)}=0.5$, $h_t^{(i)}=0.3$, $s_t^{(i)}=0.1$, $C=0.25N$, and $V=0.155N$.
Selling prices are sampled independently from $\mathcal U[1,2]$.
For $\mu_t=0.25\left[1+\cos\left(2\pi t/7\right)\right]$, demand bounds are $\underline d_t^{(i)}=\max\{0,\mu_t-0.075\}$ and $\bar d_t^{(i)}=\mu_t+0.075$.
The generator checks the constant initial-inventory constraints once and omits their rows.
With variable bounds stored separately, the LP has
\begin{equation}
    n=1+2NT\left(T+1\right),\qquad
    m_e=\frac{NT\left(T+1\right)}{2},\qquad
    m=2T+N\left[\frac{T\left(T+5\right)}{2}-1\right].
\end{equation}
At the stated settings, $\operatorname{nnz}\left(A\right)=1075N+1$.
Each instance is solved with HiGHS to obtain a primal--dual reference solution.

\subsection{\texorpdfstring{$L_\infty$}{L-infinity} Polynomial Regression}
\label{app:linf-regression}

$L_\infty$ polynomial regression minimizes the largest absolute residual over a finite set of observations.
This criterion is also called minimax or Chebyshev regression and differs from least-squares regression, which minimizes the sum of squared residuals.
Let $\left\{\left(t_i,y_i\right)\right\}_{i=1}^N$ be observations with $t_i\in[-1,1]$, and let $\left\{\phi_j\right\}_{j=1}^p$ be a basis for polynomials of degree at most $p-1$.
Write $y=\left(y_1,\ldots,y_N\right)$ and define the design matrix $\Phi\in\mathbb{R}^{N\times p}$ by $\Phi_{ij}=\phi_j\left(t_i\right)$.
The coefficient vector $\beta=\left(\beta_1,\ldots,\beta_p\right)\in\mathbb{R}^p$ solves
\begin{equation}
\label{eq:linf-regression}
    \min_{\beta\in\mathbb{R}^p}
    \left\lVert\Phi\beta-y\right\rVert_\infty
    =
    \min_{\beta\in\mathbb{R}^p}
    \max_{i=1,\ldots,N}
    \left|\sum_{j=1}^{p}\beta_j\phi_j\left(t_i\right)-y_i\right|.
\end{equation}
Introducing an epigraph variable $\tau\geq0$ gives the LP, where $\mathbf{1}\in\mathbb{R}^N$ denotes the all-ones vector.
\begin{equation}
\label{eq:linf-regression-epigraph}
\begin{aligned}
    \min_{\beta\in\mathbb{R}^p,\,\tau\geq0}\quad &\tau\\
    \text{subject to}\quad
    &-\tau\mathbf{1}\leq\Phi\beta-y\leq\tau\mathbf{1}.
\end{aligned}
\end{equation}
For the convention $Gx\geq h$ in \Eqref{eq:LP}, set $x=\left(\beta,\tau\right)$ and
\begin{equation}
\label{eq:linf-regression-standard-form}
    G=
    \begin{bmatrix}
        \Phi & \mathbf{1}\\
        -\Phi & \mathbf{1}
    \end{bmatrix},\qquad
    h=
    \begin{bmatrix}
        y\\
        -y
    \end{bmatrix},\qquad
    c=
    \begin{bmatrix}
        \mathbf{0}_p\\
        1
    \end{bmatrix}.
\end{equation}
The coefficients $\beta$ are free variables, whereas only $\tau$ has the lower bound zero.
There are no equality constraints or finite upper bounds.
Thus the generated LP has $n=p+1$ variables and $m_i=m=2N$ inequality rows.
Because $P$ is empty, the stacked matrix is $A=G$.
For continuously sampled nodes, this matrix is dense almost surely, with $\operatorname{nnz}\left(A\right)=2N\left(p+1\right)=mn$.

Let $T_k$ be the degree-$k$ Chebyshev polynomial of the first kind, defined by $T_k\left(\cos\theta\right)=\cos\left(k\theta\right)$.
For an excess degree $r\geq1$ and noise scale $\sigma\geq0$, the generator samples a target function in the Chebyshev basis and adds Gaussian observation noise.
\begin{equation}
\label{eq:linf-regression-data}
    y_i=\sum_{k=0}^{p+r-1}a_kT_k\left(t_i\right)+\sigma\epsilon_i,
    a_k,\epsilon_i\overset{\mathrm{i.i.d.}}{\sim}\mathcal{N}\left(0,1\right),
\end{equation}
We use $r=3$ and $\sigma=0.1$.
The nodes are sampled uniformly on $[-1,1]$ by default.
The generator also supports the Chebyshev nodes $t_i=\cos\left(\left(2i-1\right)\pi/\left(2N\right)\right)$ and center-clustered nodes $t_i=u_i\left|u_i\right|$ for $u_i\sim\operatorname{Unif}\left[-1,1\right]$.

The fitted basis is either the monomial basis $\phi_j\left(t\right)=t^{j-1}$ or the Chebyshev basis $\phi_j\left(t\right)=T_{j-1}\left(t\right)$.
Both bases span the same degree-$(p-1)$ polynomial space, so paired instances with the same nodes and observations have the same optimal residual $\tau^\star$ in exact arithmetic.
Their numerical conditioning can differ substantially.
The monomial design is a Vandermonde matrix and can become ill-conditioned as the degree increases, while the Chebyshev representation is typically better conditioned on $[-1,1]$.
The basis choice therefore changes the LP representation without changing the regression function class.
Unlike the preceding planted generators, we obtain the reference primal--dual solution by solving \Eqref{eq:linf-regression-epigraph} with the HiGHS dual-simplex method and refining the active constraints by least squares.

\subsection{Training set construction}
\label{app:protocol}

We generate a fixed training set for each family using the sizes and instance counts in Tables~\ref{tab:size-ladders-a} and~\ref{tab:size-ladders-b}.
Each instance is generated from its bucket seed plus its within-bucket index, and the same LPs are reused across PPO iterations.
The training sets contain 224 transport, 448 auction, 224 min-cost flow, 256 robust inventory, 512 regression, and 384 random LP instances.
We train one policy per family, select one checkpoint by held-out evaluation, and compare it on the same four unseen LPs per test size as both baselines.
Tables~\ref{tab:size-ladders-a} and~\ref{tab:size-ladders-b} list the training, validation, and test sizes, whose instances use disjoint seed ranges.
Training counts denote LPs per PPO rollout, held-out counts denote LPs per evaluation, and compared counts denote LPs per solver.
Size multipliers are relative to the largest training size and measured in $n$, except for inventory (item count $N$) and regression ($m$ at fixed $n=9$).
For inventory, $m=20+74N$ and $n=1+220N$ at a fixed horizon of 10 periods.

\begin{table}[!htbp]
\centering
\footnotesize
\setlength{\tabcolsep}{4pt}
\renewcommand{\arraystretch}{1.05}
\caption{Training, validation, and test sizes for transport, auction, and min-cost flow.}
\label{tab:size-ladders-a}
\begin{tabular}{@{}lrrrrrr@{}}
\toprule
Family & Multiplier & $m$ & $n$ & $\operatorname{nnz}(A)$ & \multicolumn{2}{c}{LPs per split}\\
\cmidrule(l){6-7}
 & & & & & Train / Held-out & Compared\\
\midrule
Transport & 0.125\,$\times$ & 120 & 3,200 & 6,400 & 64 / -- & --\\
 & 0.25\,$\times$ & 160 & 6,400 & 12,800 & 64 / 16 & 4\\
 & 0.5\,$\times$ & 240 & 12,800 & 25,600 & 64 / -- & --\\
 & 1\,$\times$ & 320 & 25,600 & 51,200 & 32 / 16 & 4\\
 & 2.5\,$\times$ & 506 & 64,000 & 128,000 & -- / 8 & 4\\
 & 5\,$\times$ & 720 & 128,000 & 256,000 & -- / 8 & 4\\
 & 10\,$\times$ & 1,012 & 256,000 & 512,000 & -- / 4 & 4\\
 & 12\,$\times$ & 1,112 & 307,200 & 614,400 & -- / -- & 4\\
 & 14\,$\times$ & 1,200 & 358,400 & 716,800 & -- / -- & 4\\
 & 16\,$\times$ & 1,280 & 409,600 & 819,200 & -- / -- & 4\\
 & 25\,$\times$ & 1,600 & 640,000 & 1,280,000 & -- / -- & 4\\
 & 36\,$\times$ & 1,920 & 921,600 & 1,843,200 & -- / -- & 4\\
 & 64\,$\times$ & 2,560 & 1,638,400 & 3,276,800 & -- / -- & 4\\
 & 100\,$\times$ & 3,200 & 2,560,000 & 5,120,000 & -- / -- & 4\\
 & 400\,$\times$ & 6,400 & 10,240,000 & 20,480,000 & -- / -- & 4\\
\addlinespace[2pt]
Auction & 0.125\,$\times$ & 800 & 3,200 & 64,000 & 128 / 16 & --\\
 & 0.25\,$\times$ & 1,600 & 6,400 & 128,000 & 128 / 16 & --\\
 & 0.5\,$\times$ & 3,200 & 12,800 & 256,000 & 128 / 16 & --\\
 & 1\,$\times$ & 6,400 & 25,600 & 512,000 & 64 / 16 & 4\\
 & 2.5\,$\times$ & 16,000 & 64,000 & 1,280,000 & -- / -- & 4\\
 & 4\,$\times$ & 25,600 & 102,400 & 2,048,000 & -- / 8 & --\\
 & 5\,$\times$ & 32,000 & 128,000 & 2,560,000 & -- / -- & 4\\
 & 10\,$\times$ & 64,000 & 256,000 & 5,120,000 & -- / -- & 4\\
 & 12\,$\times$ & 76,800 & 307,200 & 6,144,000 & -- / -- & 4\\
 & 14\,$\times$ & 89,600 & 358,400 & 7,168,000 & -- / -- & 4\\
\addlinespace[2pt]
Min-cost flow & 0.125\,$\times$ & 600 & 3,600 & 7,200 & 64 / -- & --\\
 & 0.25\,$\times$ & 1,200 & 7,200 & 14,400 & 64 / 16 & --\\
 & 0.5\,$\times$ & 2,400 & 14,400 & 28,800 & 64 / -- & --\\
 & 1\,$\times$ & 4,800 & 28,800 & 57,600 & 32 / 16 & 4\\
 & 2.5\,$\times$ & 12,000 & 72,000 & 144,000 & -- / 8 & 4\\
 & 5\,$\times$ & 24,000 & 144,000 & 288,000 & -- / 8 & 4\\
 & 10\,$\times$ & 48,000 & 288,000 & 576,000 & -- / 4 & 4\\
 & 12\,$\times$ & 57,600 & 345,600 & 691,200 & -- / -- & 4\\
\bottomrule
\end{tabular}
\end{table}

\begin{table}[!htbp]
\centering
\footnotesize
\setlength{\tabcolsep}{4pt}
\renewcommand{\arraystretch}{1.05}
\caption{Training, validation, and test sizes for inventory, $L_\infty$ regression, and random LP.}
\label{tab:size-ladders-b}
\begin{tabular}{@{}lrrrrrr@{}}
\toprule
Family & Multiplier & $m$ & $n$ & $\operatorname{nnz}(A)$ & \multicolumn{2}{c}{LPs per split}\\
\cmidrule(l){6-7}
 & & & & & Train / Held-out & Compared\\
\midrule
Robust inventory & 0.25\,$\times$ & 390 & 1,101 & 5,376 & 64 / -- & 4\\
 & 0.5\,$\times$ & 760 & 2,201 & 10,751 & 128 / 4 & 4\\
 & 1\,$\times$ & 1,500 & 4,401 & 21,501 & 64 / 4 & 4\\
 & 1.5\,$\times$ & 2,240 & 6,601 & 32,251 & -- / 4 & 4\\
 & 2\,$\times$ & 2,980 & 8,801 & 43,001 & -- / 2 & 4\\
 & 3\,$\times$ & 4,460 & 13,201 & 64,501 & -- / 2 & 4\\
 & 5\,$\times$ & 7,420 & 22,001 & 107,501 & -- / 2 & 4\\
 & 8\,$\times$ & 11,860 & 35,201 & 172,001 & -- / -- & 4\\
 & 10\,$\times$ & 14,820 & 44,001 & 215,001 & -- / -- & 4\\
\addlinespace[2pt]
$L_\infty$ regression & 0.5\,$\times$ & 120 & 9 & 1,080 & 128 / 16 & 4\\
 & 0.67\,$\times$ & 160 & 9 & 1,440 & 128 / 16 & 4\\
 & 0.83\,$\times$ & 200 & 9 & 1,800 & 128 / 16 & 4\\
 & 1\,$\times$ & 240 & 9 & 2,160 & 128 / 16 & 4\\
 & 2\,$\times$ & 480 & 9 & 4,320 & -- / 16 & 4\\
 & 5\,$\times$ & 1,200 & 9 & 10,800 & -- / 16 & 4\\
 & 6\,$\times$ & 1,440 & 9 & 12,960 & -- / -- & 4\\
 & 7\,$\times$ & 1,680 & 9 & 15,120 & -- / -- & 4\\
\addlinespace[2pt]
Random LP & 0.2\,$\times$ & 2,000 & 4,000 & 512,000 & 64 / -- & --\\
 & 0.4\,$\times$ & 4,000 & 8,000 & 1,024,000 & 128 / -- & 4\\
 & 0.8\,$\times$ & 8,000 & 16,000 & 2,048,000 & 128 / 16 & --\\
 & 1\,$\times$ & 10,000 & 20,000 & 2,560,000 & 64 / -- & 4\\
 & 2\,$\times$ & 20,000 & 40,000 & 5,120,000 & -- / 4 & 4\\
 & 3\,$\times$ & 30,000 & 60,000 & 7,680,000 & -- / -- & 4\\
 & 4\,$\times$ & 40,000 & 80,000 & 10,240,000 & -- / -- & 4\\
 & 5\,$\times$ & 50,000 & 100,000 & 12,800,000 & -- / 4 & --\\
 & 6\,$\times$ & 60,000 & 120,000 & 15,360,000 & -- / -- & 4\\
 & 10\,$\times$ & 100,000 & 200,000 & 25,600,000 & -- / 2 & 4\\
 & 15\,$\times$ & 150,000 & 300,000 & 38,400,000 & -- / -- & 4\\
\bottomrule
\end{tabular}
\end{table}
\FloatBarrier

\section{Full Results on Six Problem Classes}
\label{app:full-results}

Tables~\ref{tab:ladder-results-a} and~\ref{tab:ladder-results-b} report the complete size-generalization results for all six LP families.
Each entry is the median over four unseen LPs at the indicated size, evaluated by all three solvers.
Iterations and matrix--vector products are counted to the first iterate satisfying the common relative KKT tolerance $\epsilon_{\mathrm{tol}}=10^{-6}$.
Algorithm wall clock measures execution after compilation and excludes JIT compilation and preconditioning, as in Section~\ref{sec:experimental-setup}.
The dagger marks the largest training size, and a dash denotes failure to reach the tolerance within the iteration limit.

\begin{table}[!htbp]
\centering
\scriptsize
\setlength{\tabcolsep}{3.4pt}
\renewcommand{\arraystretch}{1.02}
\caption{Full test results for transport, auction, and min-cost flow (four-LP medians).}
\label{tab:ladder-results-a}
\begin{tabular}{@{}ll rrr rrr rrr@{}}
\toprule
& & \multicolumn{3}{c}{Iterations to $\epsilon_{\mathrm{tol}}$} & \multicolumn{3}{c}{Matrix--vector products} & \multicolumn{3}{c}{Algorithm wall clock (ms)}\\
\cmidrule(lr){3-5}\cmidrule(lr){6-8}\cmidrule(l){9-11}
Family & $(m, n)$ & Ours & HT & MPAX & Ours & HT & MPAX & Ours & HT & MPAX\\
\midrule
Transport & (160, 6,400) & \textbf{146} & 488 & 550 & \textbf{294} & 979 & 1,158 & \textbf{46} & 119 & 76\\
 & (320, 25,600)$^\dagger$ & \textbf{164} & 633 & 548 & \textbf{329} & 1,268 & 1,150 & \textbf{39} & 107 & 87\\
 & (506, 64,000) & \textbf{155} & 674 & 540 & \textbf{312} & 1,349 & 1,134 & \textbf{46} & 143 & 96\\
 & (720, 128,000) & \textbf{150} & 628 & 526 & \textbf{302} & 1,257 & 1,108 & \textbf{48} & 146 & 122\\
 & (1,012, 256,000) & \textbf{140} & 554 & 518 & \textbf{282} & 1,111 & 1,090 & \textbf{52} & 154 & 156\\
 & (1,112, 307,200) & \textbf{158} & 628 & 528 & \textbf{317} & 1,259 & 1,110 & \textbf{57} & 173 & 170\\
 & (1,200, 358,400) & \textbf{152} & 621 & 527 & \textbf{307} & 1,244 & 1,108 & \textbf{61} & 192 & 188\\
 & (1,280, 409,600) & \textbf{148} & 712 & 524 & \textbf{299} & 1,427 & 1,101 & \textbf{66} & 247 & 203\\
 & (1,600, 640,000) & \textbf{153} & 533 & 536 & \textbf{308} & 1,068 & 1,128 & \textbf{102} & 286 & 308\\
 & (1,920, 921,600) & \textbf{152} & 594 & 538 & \textbf{305} & 1,190 & 1,130 & \textbf{135} & 447 & 445\\
 & (2,560, 1,638,400) & \textbf{149} & 676 & 527 & \textbf{300} & 1,354 & 1,110 & \textbf{248} & 965 & 830\\
 & (3,200, 2,560,000) & \textbf{149} & 608 & 532 & \textbf{300} & 1,217 & 1,120 & \textbf{385} & 1,399 & 1,445\\
 & (6,400, 10,240,000) & \textbf{200} & 570 & 524 & \textbf{403} & 1,142 & 1,105 & \textbf{1,998} & 5,222 & 7,334\\
\addlinespace[2pt]
Auction & (6,400, 25,600)$^\dagger$ & \textbf{190} & 928 & 980 & \textbf{381} & 1,859 & 2,025 & \textbf{70} & 256 & 198\\
 & (16,000, 64,000) & \textbf{190} & 924 & 1,108 & \textbf{382} & 1,851 & 2,284 & \textbf{81} & 293 & 324\\
 & (32,000, 128,000) & \textbf{190} & 923 & 994 & \textbf{382} & 1,848 & 2,055 & \textbf{84} & 316 & 479\\
 & (64,000, 256,000) & \textbf{192} & 918 & 1,044 & \textbf{387} & 1,838 & 2,158 & \textbf{114} & 433 & 895\\
 & (76,800, 307,200) & \textbf{192} & 924 & 1,046 & \textbf{387} & 1,849 & 2,162 & \textbf{118} & 453 & 1,143\\
 & (89,600, 358,400) & \textbf{192} & 926 & 1,066 & \textbf{387} & 1,853 & 2,198 & \textbf{134} & 521 & 1,352\\
\addlinespace[2pt]
Min-cost flow & (4,800, 28,800)$^\dagger$ & \textbf{164} & 302 & 440 & \textbf{329} & 606 & 928 & \textbf{49} & 68 & 76\\
 & (12,000, 72,000) & \textbf{180} & 319 & 448 & \textbf{363} & 640 & 944 & \textbf{63} & 85 & 88\\
 & (24,000, 144,000) & \textbf{234} & 356 & 444 & \textbf{469} & 715 & 939 & \textbf{80} & 94 & 109\\
 & (48,000, 288,000) & \textbf{224} & 374 & 446 & \textbf{451} & 749 & 938 & \textbf{91} & 120 & 139\\
 & (57,600, 345,600) & \textbf{242} & 370 & 482 & \textbf{485} & 741 & 1,010 & \textbf{115} & 142 & 147\\
\bottomrule
\end{tabular}
\par\smallskip
\parbox{\linewidth}{\scriptsize Ours: \textcolor{black}{GALLOP}; HT: GA-PDHG (HT). Bold: best per metric. $^\dagger$: largest training size. A dash denotes failure to reach $\epsilon_{\mathrm{tol}}=10^{-6}$ within the iteration limit.}
\end{table}

\begin{table}[!htbp]
\centering
\scriptsize
\setlength{\tabcolsep}{3.4pt}
\renewcommand{\arraystretch}{1.02}
\caption{Full test results for inventory, $L_\infty$ regression, and random LP (four-LP medians).}
\label{tab:ladder-results-b}
\begin{tabular}{@{}ll rrr rrr rrr@{}}
\toprule
& & \multicolumn{3}{c}{Iterations to $\epsilon_{\mathrm{tol}}$} & \multicolumn{3}{c}{Matrix--vector products} & \multicolumn{3}{c}{Algorithm wall clock (ms)}\\
\cmidrule(lr){3-5}\cmidrule(lr){6-8}\cmidrule(l){9-11}
Family & $(m, n)$ & Ours & HT & MPAX & Ours & HT & MPAX & Ours & HT & MPAX\\
\midrule
Robust inventory & (390, 1,101) & \textbf{1,426} & 2,892 & 3,774 & \textbf{2,853} & 5,787 & 7,682 & \textbf{259} & 397 & 309\\
 & (760, 2,201) & \textbf{1,250} & 4,242 & 3,072 & \textbf{2,503} & 8,485 & 6,256 & \textbf{260} & 659 & 482\\
 & (1,500, 4,401)$^\dagger$ & \textbf{1,905} & 4,266 & 3,875 & \textbf{3,812} & 8,535 & 7,872 & 507 & 915 & \textbf{432}\\
 & (2,240, 6,601) & \textbf{2,140} & 10,294 & 5,888 & \textbf{4,282} & 20,591 & 11,934 & 635 & 2,529 & \textbf{591}\\
 & (2,980, 8,801) & \textbf{2,311} & 7,768 & 5,348 & \textbf{4,624} & 15,537 & 10,854 & \textbf{570} & 1,494 & 661\\
 & (4,460, 13,201) & \textbf{2,874} & 25,798 & 6,121 & \textbf{5,749} & 51,599 & 12,391 & \textbf{858} & 5,978 & 895\\
 & (7,420, 22,001) & \textbf{3,833} & 12,103 & 6,167 & \textbf{7,668} & 24,208 & 12,495 & 1,164 & 3,034 & \textbf{975}\\
 & (11,860, 35,201) & \textbf{6,308} & 13,499 & 8,588 & \textbf{12,618} & 27,000 & 17,356 & 1,920 & 3,209 & \textbf{1,505}\\
 & (14,820, 44,001) & \textbf{7,832} & -- & 11,581 & \textbf{15,666} & -- & 23,370 & 2,239 & -- & \textbf{1,834}\\
\addlinespace[2pt]
$L_\infty$ regression & (120, 9) & \textbf{190} & 465 & 702 & \textbf{383} & 932 & 1,466 & \textbf{28} & 42 & 69\\
 & (160, 9) & \textbf{169} & 554 & 907 & \textbf{340} & 1,110 & 1,876 & \textbf{25} & 51 & 81\\
 & (200, 9) & \textbf{127} & 533 & 740 & \textbf{256} & 1,068 & 1,544 & \textbf{20} & 49 & 72\\
 & (240, 9)$^\dagger$ & \textbf{150} & 586 & 834 & \textbf{303} & 1,175 & 1,734 & \textbf{24} & 54 & 82\\
 & (480, 9) & \textbf{270} & 1,060 & 786 & \textbf{543} & 2,121 & 1,637 & \textbf{44} & 113 & 72\\
 & (1,200, 9) & \textbf{236} & 673 & 760 & \textbf{475} & 1,348 & 1,583 & \textbf{41} & 78 & 69\\
 & (1,440, 9) & \textbf{308} & 741 & 764 & \textbf{619} & 1,484 & 1,590 & \textbf{54} & 92 & 76\\
 & (1,680, 9) & \textbf{358} & 752 & 782 & \textbf{718} & 1,505 & 1,626 & \textbf{62} & 93 & 78\\
\addlinespace[2pt]
Random LP & (4,000, 8,000) & \textbf{264} & 424 & 706 & \textbf{531} & 850 & 1,468 & \textbf{87} & 109 & 221\\
 & (10,000, 20,000)$^\dagger$ & \textbf{284} & 440 & 664 & \textbf{570} & 883 & 1,384 & \textbf{110} & 140 & 408\\
 & (20,000, 40,000) & \textbf{270} & 416 & 650 & \textbf{541} & 834 & 1,359 & \textbf{125} & 164 & 773\\
 & (30,000, 60,000) & \textbf{280} & 421 & 663 & \textbf{561} & 844 & 1,384 & \textbf{151} & 199 & 1,175\\
 & (40,000, 80,000) & \textbf{264} & 429 & 666 & \textbf{530} & 860 & 1,388 & \textbf{169} & 244 & 1,612\\
 & (60,000, 120,000) & \textbf{304} & 426 & 774 & \textbf{609} & 854 & 1,609 & \textbf{243} & 310 & 2,443\\
 & (100,000, 200,000) & \textbf{314} & 455 & 711 & \textbf{629} & 912 & 1,481 & \textbf{363} & 490 & 5,283\\
 & (150,000, 300,000) & \textbf{290} & 428 & 702 & \textbf{582} & 859 & 1,468 & \textbf{450} & 626 & 7,202\\
\bottomrule
\end{tabular}
\par\smallskip
\parbox{\linewidth}{\scriptsize Ours: \textcolor{black}{GALLOP}; HT: GA-PDHG (HT). Bold: best per metric. $^\dagger$: largest training size. A dash denotes failure to reach $\epsilon_{\mathrm{tol}}=10^{-6}$ within the iteration limit.}
\end{table}
\FloatBarrier

\section{Implementation Details}
\label{app:implementation-details}

\subsection{Network architecture and action distributions}
\label{app:controller-policy}

The \textcolor{black}{GALLOP} policy uses separate tanh MLPs for the restart decision, history coefficients, relaxation and anchor coefficients, step size, and primal weight.
A separate network estimates the value function.
Each network has two hidden layers of width 64.
The restart head produces $p_t^e\in(0,1)$ through a sigmoid and samples
\begin{equation}
    e_t\sim\mathrm{Bernoulli}(p_t^e).
\end{equation}
Conditional on $o_t$, the continuous heads sample independent Gaussian latents with the means $\mu_t^{\alpha^x},\mu_t^{\alpha^y},\boldsymbol\mu_t^\beta,\mu_t^\zeta,\mu_t^\eta,\mu_t^\omega$ output by their respective networks; all but $\boldsymbol\mu_t^\beta\in\mathbb R^k$ are scalars.
Their standard deviations $\sigma_\alpha,\sigma_\beta,\sigma_\zeta,\sigma_\eta,\sigma_\omega>0$ are learned globally and do not depend on $o_t$.
Their logarithms are optimized with the corresponding actor and clipped to group-specific lower bounds and a common upper bound when the distributions are evaluated.
Appendix~\ref{app:experiment-implementation} gives the initialization values and bounds used for each family.
The primal and dual relaxation latents share $\sigma_\alpha$, and all history coordinates share $\sigma_\beta$.
We restrict the learned coefficients to $\alpha_t^x,\alpha_t^y>0$ to focus the policy search on extrapolation beyond the PDHG output, motivated by potential acceleration rather than damping through underrelaxation.
Exponential transformations enforce positive relaxation coefficients and a positive anchor multiplier,
\begin{align}
    u_t^{\alpha^x}&\sim\mathcal N\left(\mu_t^{\alpha^x},\sigma_\alpha^2\right),
    &\alpha_t^x&=\exp\left(u_t^{\alpha^x}\right),\\
    u_t^{\alpha^y}&\sim\mathcal N\left(\mu_t^{\alpha^y},\sigma_\alpha^2\right),
    &\alpha_t^y&=\exp\left(u_t^{\alpha^y}\right),\\
    u_t^\zeta&\sim\mathcal N\left(\mu_t^\zeta,\sigma_\zeta^2\right),
    &\widetilde\zeta_t&=\exp\left(u_t^\zeta\right).
\end{align}
The effective anchor coefficient $\zeta_t$ is computed by \Eqref{eq:ga-pdhg-anchor-parameterization}, with $\zeta_t=0$ when $z_\star=z_t$.
During training, the history coefficients are sampled from an unbounded Gaussian distribution,
\begin{equation}
    \boldsymbol\beta_t\sim\mathcal N\left(\boldsymbol\mu_t^\beta,\sigma_\beta^2I_k\right).
\end{equation}
At every iteration, including restarts, the step-size sample is transformed as
\begin{equation}
\label{eq:learned_eta}
    u_t^\eta\sim\mathcal N\left(\mu_t^\eta,\sigma_\eta^2\right),\qquad
    \eta_t=\frac{0.1+4.9\sig\left(u_t^\eta\right)}{\|A\|_2}.
\end{equation}
Here $\sig(u):=1/(1+e^{-u})$ is the sigmoid.
The denominator estimates the spectral norm of the matrix supplied to PDHG.
Sparse instances use a converged spectral-norm calculation during both training and evaluation; dense training uses 100 power iterations, while dense evaluation uses a converged calculation.
The initialization step uses $0.998/\|A\|_2$.
The learned step size can exceed $1/\|A\|_2$ and therefore does not enforce the fixed-step nonexpansiveness condition.
At every iteration, including non-restart transitions, the primal-weight sample sets an absolute target relative to $\omega_0$,
\begin{equation}
\label{eq:learned_omega}
    u_t^\omega\sim\mathcal N\left(\mu_t^\omega,\sigma_\omega^2\right),\qquad
    \omega_t=\omega_0\exp\left(L_\omega\tanh\left(u_t^\omega\right)\right).
\end{equation}
Here $L_\omega=2$ for auction and $L_\omega=4.6$ for the other five families.
This weight is applied in the next PDHG map regardless of $e_t$; it does not multiply the previous weight.
The rollout records the raw action
\[
    a_t=(e_t,u_t^{\alpha^x},u_t^{\alpha^y},\boldsymbol\beta_t,u_t^\zeta,u_t^\eta,u_t^\omega).
\]
Let $a_t^{\mathrm{act}}$ consist of $e_t$, $u_t^\eta$, $u_t^\omega$, and the acceleration groups enabled when $e_t=0$.
Marginalizing the unused draws gives the density of this effective action,
\begin{equation}
\label{eq:masked_density_gen}
\begin{aligned}
    \pi_{\theta,\mathrm{act}}\left(a_t^{\mathrm{act}}\mid o_t\right)
    ={}&(p_t^e)^{e_t}(1-p_t^e)^{1-e_t}
    \mathcal N\left(u_t^\eta;\mu_t^\eta,\sigma_\eta^2\right)
    \mathcal N\left(u_t^\omega;\mu_t^\omega,\sigma_\omega^2\right)\\
    &\times\left[\begin{aligned}
    &\mathcal N\left(u_t^{\alpha^x};\mu_t^{\alpha^x},\sigma_\alpha^2\right)
    \mathcal N\left(u_t^{\alpha^y};\mu_t^{\alpha^y},\sigma_\alpha^2\right)\\
    &\times\mathcal N\left(\boldsymbol\beta_t;\boldsymbol\mu_t^\beta,\sigma_\beta^2I_k\right)
    \mathcal N\left(u_t^\zeta;\mu_t^\zeta,\sigma_\zeta^2\right)
    \end{aligned}\right]^{1-e_t}.
\end{aligned}
\end{equation}
The step-size and primal-weight factors are active on every transition, and the acceleration factors only when $e_t=0$.
Inactive factors contribute zero to the log-density.
PPO evaluates the density on the raw latents, so the output transformations require no Jacobian correction in the likelihood ratio.

\paragraph{Deterministic inference.}
At inference, $e_t=\mathbf1[p_t^e\geq0.5]$, and every Gaussian latent is replaced by its mean before applying the same transformations.
The observation map and action transforms are unchanged, and the trained parameters remain fixed throughout the solve.
When $e_t=0$, the transformed coefficients form the accelerated input in \Eqref{eq:gen_acc_iter}; when $e_t=1$, the input is $z_t$ and the history is reset after the PDHG step.
The step size and primal weight are applied on both branches at every iteration.
The critic and Gaussian standard deviations are needed for training, but do not determine the inference actions.

\subsection{\textcolor{black}{GALLOP} observations}
\label{app:controller-observations}

For $t\geq1$, the observation $o_t$ is evaluated before $a_t$ using the carried weight $\omega_{t-1}$.
The initial weight $\omega_0$ is used for the initialization step that produces $z_1$.
It is $\|c\|_2/\|b\|_2$ in solver coordinates when both norms are positive, and $1$ otherwise.
The features below follow the convergence state, restart context, and local geometry described in Section~\ref{sec:learning-controller}.

\paragraph{Convergence state.}
The primal residual, dual residual, and duality-gap diagnostic in original LP coordinates are defined in \Eqref{eq:residuals}.
The sign cone is $\Lambda:=\prod_{i=1}^n\Lambda_i$, where $\Lambda_i=\mathbb R$ if both bounds are finite, $\mathbb R_{\geq0}$ if only $l_i$ is finite, $\mathbb R_{\leq0}$ if only $u_i$ is finite, and $\{0\}$ if neither is finite.
With elementwise $[a]^\pm:=\max\{\pm a,0\}$, the box minimum in $r_g\left(z\right)$ equals $l^\top[\bar c]^+-u^\top[\bar c]^-$, with infinite-bound terms omitted.
When $y\in Y$ and $r_d\left(z\right)=0$, $r_g\left(z\right)$ is the primal--dual objective gap. Otherwise it is a finite diagnostic used together with the feasibility residuals.
The corresponding log normalized metrics are
\begin{equation}
\label{eq:primal_residual}
    p(z):=\log\left(
    \frac{\max\left\{\|r_p(z)\|_\infty,\|[l-x]^+\|_\infty,\|[x-u]^+\|_\infty\right\}}
    {1+\max\left\{\|b\|_\infty,\|Ax\|_\infty\right\}}\right).
\end{equation}
\begin{equation}
\label{eq:dual_residual}
    d(z):=\log\left(
    \frac{\max\left\{\|r_d(z)\|_\infty,\|[-y^{\mathrm{in}}]^+\|_\infty\right\}}
    {1+\max\left\{\|c\|_\infty,\|A^\top y\|_\infty\right\}}\right).
\end{equation}
\begin{equation}
\label{eq:duality_gap}
    g(z):=\log\left(
    \frac{|r_g(z)|}
    {1+\max\left\{|c^\top x|,|c^\top x-r_g(z)|\right\}}\right).
\end{equation}
All three metrics are evaluated in the original LP coordinates, with log arguments floored at $10^{-20}$.
The bound and dual-sign violation terms vanish on $X\times Y$.
Thus these formulas also apply to the initial anchor $z_\star=0$, which may lie outside $X$.

For the controller, fixed-point progress uses the executed direction $z_t-\hat z_t$ in solver coordinates, measured in a fixed metric,
\begin{equation}
\label{eq:gen_fpr}
    M_t^{\mathrm{fp}}
    :=\|z_t-\hat z_t\|_{Q\left(1,0.998/\|A\|_2\right)}.
\end{equation}
The measurement metric uses the scaled $A$ when preconditioning is enabled.
Its spectral norm is computed by dense SVD for dense training batches and by a converged calculation for sparse batches and evaluation, independently of the dense training step-size estimate.
For $t\geq1$, it measures the step executed with $\left(\omega_{t-1},\eta_{t-1}\right)$.
The observation includes progress relative to the initial step, the restart baseline $M_\star$, and the preceding step,
\begin{equation}
    \log\frac{\max\{M_t^{\mathrm{fp}},10^{-20}\}}{\max\{M_1^{\mathrm{fp}},10^{-20}\}},\qquad
    \log\frac{\max\{M_t^{\mathrm{fp}},10^{-20}\}}{\max\{M_\star,10^{-20}\}},\qquad
    \log\frac{\max\{M_t^{\mathrm{fp}},10^{-20}\}}{\max\{M_{t-1}^{\mathrm{fp}},10^{-20}\}}.
\end{equation}
It also includes the one-step changes in the KKT metrics,
\begin{equation}
    p(z_t)-p(z_{t-1}),\qquad d(z_t)-d(z_{t-1}),\qquad g(z_t)-g(z_{t-1}).
\end{equation}
The one-step features are evaluated only for $t\geq2$ and are set to zero at $t=1$.
The initial restart baseline is $M_\star=M_1^{\mathrm{fp}}$.

\paragraph{Restart context.}
The policy observes the log restart age and global iteration count, $\log(s+1)$ and $\log t$, and the KKT progress from the anchor,
\begin{equation}
    p(z_t)-p(z_\star),\qquad d(z_t)-d(z_\star),\qquad g(z_t)-g(z_\star).
\end{equation}
At initialization, the anchor metrics are evaluated at $z_\star=0$.
After a restart step, $z_\star\leftarrow z_{t+1}$ and $M_\star\leftarrow M_{t+1}^{\mathrm{fp}}$, so the anchor-relative KKT differences are zero.
Primal--dual balance, the relative weight, and its one-step change are measured by
\begin{equation}
\label{eq:gen_weight_observation}
    b_t^e:=\log\left(
    \frac{\omega_{t-1}\|x_t-x_\star\|_2+10^{-20}}
    {\|y_t-y_\star\|_2+10^{-20}}\right),
    \qquad \log\left(\frac{\omega_{t-1}}{\omega_0}\right),\qquad\log\left(\frac{\omega_{t-1}}{\omega_{t-2}}\right).
\end{equation}
The weight-change feature is set to zero at $t=1$.
The displacements in these features use the solver coordinates.
KKT residuals evaluated at $\hat z_t$ are not supplied to the policy.

\paragraph{Local geometry.}
Partition the history directions into primal and dual blocks,
\begin{equation}
    \mathbf K_t=\begin{pmatrix}\mathbf K_t^x\\\mathbf K_t^y\end{pmatrix},\qquad
    \mathbf K_t^x\in\mathbb R^{n\times k},\quad
    \mathbf K_t^y\in\mathbb R^{m\times k}.
\end{equation}
For $z=\left(x^\top,y^\top\right)^\top$, define the weighted norm
\begin{equation}
\label{eq:weight_norm}
    \lVert z\rVert_\omega^2:=\omega\lVert x\rVert_2^2+\omega^{-1}\lVert y\rVert_2^2.
\end{equation}
Using this norm and $\epsilon_{\mathrm{geom}}=10^{-8}$, the normalized history Gram matrix and its alignment with the current direction are
\begin{equation}
\label{eq:gen_history_geometry}
\begin{aligned}
    \Gamma_t&:=\frac{
        \omega_{t-1}(\mathbf K_t^x)^\top\mathbf K_t^x
        +\omega_{t-1}^{-1}(\mathbf K_t^y)^\top\mathbf K_t^y
    }{\|z_t-\hat z_t\|_{\omega_{t-1}}^2+\epsilon_{\mathrm{geom}}^2}
    \in\mathbb R^{k\times k},\\
    c_t^{\mathrm{hist}}&:=\frac{
        \omega_{t-1}(\mathbf K_t^x)^\top\left(x_t-\hat x_t\right)
        +\omega_{t-1}^{-1}(\mathbf K_t^y)^\top\left(y_t-\hat y_t\right)
    }{\|z_t-\hat z_t\|_{\omega_{t-1}}^2+\epsilon_{\mathrm{geom}}^2}
    \in\mathbb R^k.
\end{aligned}
\end{equation}
The anchor features measure its relative magnitude and weighted alignment with the current direction,
\begin{equation}
\label{eq:gen_anchor_geometry}
\begin{aligned}
    \rho_t^{\mathrm{anc}}&:=\log\frac{\|z_\star-z_t\|_{\omega_{t-1}}+\epsilon_{\mathrm{geom}}}
        {\|z_t-\hat z_t\|_{\omega_{t-1}}+\epsilon_{\mathrm{geom}}},\\
    c_t^{\mathrm{anc}}&:=\frac{\omega_{t-1}(x_\star-x_t)^\top\left(x_t-\hat x_t\right)+\omega_{t-1}^{-1}(y_\star-y_t)^\top\left(y_t-\hat y_t\right)}
        {\left(\|z_\star-z_t\|_{\omega_{t-1}}+\epsilon_{\mathrm{geom}}\right)
        \left(\|z_t-\hat z_t\|_{\omega_{t-1}}+\epsilon_{\mathrm{geom}}\right)}.
\end{aligned}
\end{equation}
The history quantities undergo an elementwise signed logarithm before entering the policy,
\begin{equation}
\label{eq:gen_geometry_features}
\begin{aligned}
    &\operatorname{vech}\left(\operatorname{sign}(\Gamma_t)\log\left(1+|\Gamma_t|\right)\right),\qquad
    \operatorname{sign}(c_t^{\mathrm{hist}})\log\left(1+|c_t^{\mathrm{hist}}|\right),\\
    &\rho_t^{\mathrm{anc}},\qquad c_t^{\mathrm{anc}}.
\end{aligned}
\end{equation}
Here $\operatorname{vech}$ stacks the upper-triangular entries, including the diagonal, in row-major order. The sign, absolute value, logarithm, and their products act elementwise.
The observation has $19+k+k(k+1)/2$ entries, giving 33 inputs for $k=4$.
These inputs are passed directly to the networks without running normalization or input clipping.

\subsection{Policy optimization}
\label{app:controller-optimization}

At each PPO iteration, the parameters $(\theta_{\mathrm{old}},\phi_{\mathrm{old}})$ are frozen while one episode is collected per LP instance.
For an episode of $T$ transitions, generalized advantage estimation~\citep{schulman2016gae} uses discount factor $\gamma=1$ and GAE coefficient $\kappa=0.99$,
\begin{equation}
\label{eq:gae}
\begin{aligned}
\delta_t &:= r_t+\gamma\left(1-\mathrm{success}_t\right)V^{\phi_{\mathrm{old}}}(o_{t+1})-V^{\phi_{\mathrm{old}}}(o_t),\\
\hat \Psi_t &:= \sum_{j=0}^{T-t}(\gamma\kappa)^j\,\delta_{t+j},\qquad 1\leq t\leq T.
\end{aligned}
\end{equation}
Here $\mathrm{success}_t:=\mathbf1\left[M^{\mathrm{rel}}(z_{t+1})\leq\log\epsilon_{\mathrm{tol}}\right]$.
The sum ends with the episode. A successful episode has terminal value zero; truncation at $T_{\max}$ retains $V^{\phi_{\mathrm{old}}}\left(o_{T+1}\right)$ in the final temporal-difference residual.
The critic target and the advantage normalized over the full rollout buffer are
\begin{equation}
\label{eq:gae_targets}
\hat R_t := \hat \Psi_t+V^{\phi_{\mathrm{old}}}(o_t),
\qquad
\tilde \Psi_t := \frac{\hat \Psi_t-\mathrm{mean}(\hat \Psi)}{\mathrm{std}(\hat \Psi) + 10^{-8}}.
\end{equation}
Here $\mathrm{mean}(\hat\Psi)$ and $\mathrm{std}(\hat\Psi)$ are the empirical mean and population standard deviation of all valid advantages in $\mathcal B$.
Writing $\hat{\mathbb E}_t$ for the empirical average over that buffer, the critic loss is
\begin{equation}
\label{eq:ppo_critic_loss}
\mathcal{L}_V(\phi)
:=\hat{\mathbb{E}}_t\!\left[
\left(V^\phi(o_t)-\hat R_t\right)^2
\right].
\end{equation}

\paragraph{Groupwise PPO.}
Let $\pi_{\theta,e}$ and $\pi_{\theta,\beta}$ denote the restart and history densities, $\pi_{\theta,\alpha\zeta}$ the product of the relaxation and anchor Gaussian densities, and $\pi_{\theta,\eta},\pi_{\theta,\omega}$ the step-size and primal-weight latent densities.
Each density below is evaluated at its recorded component of $a_t$, conditional on the recorded $o_t$; these arguments are suppressed.
For the action groups $\mathcal G=\{e,\beta,\alpha\zeta,\eta,\omega\}$, the masked log-densities are
\begin{equation}
\label{eq:group_densities}
    \boldsymbol\ell_t:=\left(\log\pi_{\theta,e},
    (1-e_t)\log\pi_{\theta,\beta},
    (1-e_t)\log\pi_{\theta,\alpha\zeta},
    \log\pi_{\theta,\eta},
    \log\pi_{\theta,\omega}\right)^\top.
\end{equation}
For $j\in\mathcal G$, let $\ell_t^j$ be the corresponding entry of $\boldsymbol\ell_t$ and $\ell_{t,\mathrm{old}}^j$ its value under $\theta_{\mathrm{old}}$.
Each group uses its own likelihood ratio and clipped surrogate, with clipping parameter $\varepsilon^{\mathrm{clip}}=0.2$,
\begin{equation}
\label{eq:ppo_group_objective}
\begin{aligned}
    \rho_t^j&:=\exp\left(\ell_t^j-\ell_{t,\mathrm{old}}^j\right),\\
    \mathcal J_j(\theta)&:=\hat{\mathbb E}_t\left[
    \min\left(\rho_t^j\tilde\Psi_t,
    \clip\left(\rho_t^j,1-\varepsilon^{\mathrm{clip}},1+\varepsilon^{\mathrm{clip}}\right)\tilde\Psi_t\right)\right].
\end{aligned}
\end{equation}
Here $\clip(v,a,b):=\min\{b,\max\{a,v\}\}$.
The five surrogates are summed; the advantages and recorded old log-densities are held fixed during optimization.

On restart rows, the acceleration groups have $\rho_t^j=1$ and contribute no policy gradient.
The restart, step-size, and primal-weight surrogates use all valid transitions in the common minibatches and their full-trajectory, globally normalized advantages $\tilde\Psi_t$.
The primal-weight actor is therefore also updated when the rollout contains no restarts.
Each actor and the critic has its own gradient clipping and Adam state.
The critic minimizes $c_V\mathcal L_V$, with $c_V$ specified in Appendix~\ref{app:experiment-implementation}.

\paragraph{Auxiliary restart supervision.}
Before the action is sampled, the conventional sufficient-progress, deterioration, and epoch-length criteria provide a binary label $y_t^{\mathrm{prior}}$.
The label is one if $M_t^{\mathrm{fp}}\leq0.2M_\star$, if both $M_t^{\mathrm{fp}}\leq0.6M_\star$ and $M_t^{\mathrm{fp}}>M_{t-1}^{\mathrm{fp}}$, or if $s\geq0.2t$.
At $t=1$, the preceding residual in this test is taken to equal $M_1^{\mathrm{fp}}$.
It is zero otherwise and is always set to zero when $s=0$.
These tests use the policy trajectory and do not execute a restart.
The auxiliary loss and the resulting actor loss are
\begin{equation}
\label{eq:restart_prior}
\begin{aligned}
    \mathcal L_{\mathrm{prior}}&:=-\hat{\mathbb E}_t\left[
    y_t^{\mathrm{prior}}\log p_t^e+(1-y_t^{\mathrm{prior}})\log(1-p_t^e)\right],\\
    \mathcal L_{\mathrm{actor}}&:=-\sum_{j\in\mathcal G}\mathcal J_j
    +\lambda_{\mathrm{prior}}\mathcal L_{\mathrm{prior}}.
\end{aligned}
\end{equation}
The coefficient $\lambda_{\mathrm{prior}}$ is fixed within each run; its family-specific value is given in Appendix~\ref{app:experiment-implementation}.
The binary cross-entropy term updates only the restart actor; its pre-action labels are held fixed during optimization.
The rollout stores the observations, raw actions, masked log-densities, rewards, success indicators, critic values, and pre-action labels,
\begin{equation}
\label{eq:gen_rollout_record}
    \tau=\left\{\left(o_t,a_t,\boldsymbol\ell_t,r_t,
    \mathrm{success}_t,V^\phi(o_t),y_t^{\mathrm{prior}}\right)\right\}_{t=1}^{T}.
\end{equation}

\subsection{Preconditioning}
\label{app:preconditioning}

We apply 10 iterations of Ruiz scaling~\citep{ruiz2001scaling}, followed by Pock--Chambolle scaling with exponent $1$~\citep{pock2011diagonal}.
Let $D_r\in\mathbb R^{m\times m}$ and $D_c\in\mathbb R^{n\times n}$ be the resulting positive diagonal row and column scalings, with $D_r=\begin{pmatrix}D_P&0\\0&D_G\end{pmatrix}$ partitioned by equality and inequality constraints.
The scaled problem data and their relation to the original variables are
\begin{equation}
\label{eq:scaled_problem_data}
\begin{aligned}
    A^{\mathrm{sc}} &= D_r A D_c,
    &\qquad b^{\mathrm{sc}} &= D_r b,
    &\qquad c^{\mathrm{sc}} &= D_c c, \\
    P^{\mathrm{sc}} &= D_P P D_c,
    & q^{\mathrm{sc}} &= D_P q, \\
    G^{\mathrm{sc}} &= D_G G D_c,
    & h^{\mathrm{sc}} &= D_G h, \\
    l^{\mathrm{sc}} &= D_c^{-1}l,
    & u^{\mathrm{sc}} &= D_c^{-1}u.
\end{aligned}
\end{equation}
\begin{equation}
\label{eq:scaled_variables}
    x = D_c x^{\mathrm{sc}}, \qquad
    y = D_r y^{\mathrm{sc}}.
\end{equation}
In this subsection, superscript $\mathrm{sc}$ explicitly identifies scaled quantities, and unsuperscripted data and $z=\left(x^\top,y^\top\right)^\top$ refer to the original LP.
The solver notation elsewhere omits $\mathrm{sc}$; calls to $p,d,g$, and $M^{\mathrm{rel}}$ include the coordinate conversion in \Eqref{eq:scaled_variables}.
PDHG operates in scaled coordinates, while convergence metrics, rewards, and termination are evaluated for the original LP.
The residual transformations and gap invariance are
\begin{equation}
\label{eq:scaled_residual_transform}
    r_p^{\mathrm{sc}} = D_r r_p, \qquad
    r_d^{\mathrm{sc}} = D_c r_d,
\end{equation}
\begin{equation}
\label{eq:scaled_gap_invariant}
    r_g^{\mathrm{sc}} = r_g.
\end{equation}
As in \Eqref{eq:residuals}, infinite-bound terms are omitted from the computed gap.
Original residual norms can be recovered from scaled residuals,
\begin{equation}
\label{eq:original_residual_from_scaled}
\begin{aligned}
    \lVert r_p\rVert_\infty
    &= \lVert D_r^{-1} r_p^{\mathrm{sc}}\rVert_\infty
    = \max_i
    \frac{\left|(r_p^{\mathrm{sc}})_i\right|}{(D_r)_{ii}}, \\
    \lVert r_d\rVert_\infty
    &= \lVert D_c^{-1} r_d^{\mathrm{sc}}\rVert_\infty
    = \max_j
    \frac{\left|(r_d^{\mathrm{sc}})_j\right|}{(D_c)_{jj}}.
\end{aligned}
\end{equation}
The normalized primal and dual metrics, including bound and dual-sign violations, are therefore
\begin{equation}
\label{eq:preconditioned_primal_residual}
    p(z)=\log\left(
    \frac{\max\left\{
    \|D_r^{-1}r_p^{\mathrm{sc}}\|_\infty,
    \|D_c[l^{\mathrm{sc}}-x^{\mathrm{sc}}]^+\|_\infty,
    \|D_c[x^{\mathrm{sc}}-u^{\mathrm{sc}}]^+\|_\infty
    \right\}}
    {1+\max\left\{\|b\|_\infty,\|D_r^{-1}A^{\mathrm{sc}}x^{\mathrm{sc}}\|_\infty\right\}}
    \right).
\end{equation}
\begin{equation}
\label{eq:preconditioned_dual_residual}
    d(z)=\log\left(
    \frac{\max\left\{
    \|D_c^{-1}r_d^{\mathrm{sc}}\|_\infty,
    \|D_G[-y^{\mathrm{sc},\mathrm{in}}]^+\|_\infty
    \right\}}
    {1+\max\left\{\|c\|_\infty,\|D_c^{-1}(A^{\mathrm{sc}})^\top y^{\mathrm{sc}}\|_\infty\right\}}
    \right).
\end{equation}
The primal objective is also invariant, $(c^{\mathrm{sc}})^\top x^{\mathrm{sc}}=c^\top x$, so the gap metric is
\begin{equation}
\label{eq:preconditioned_gap}
    g(z)=\log\left(
    \frac{|r_g^{\mathrm{sc}}(z^{\mathrm{sc}})|}
    {1+\max\left\{|(c^{\mathrm{sc}})^\top x^{\mathrm{sc}}|,|(c^{\mathrm{sc}})^\top x^{\mathrm{sc}}-r_g^{\mathrm{sc}}(z^{\mathrm{sc}})|\right\}}
    \right).
\end{equation}
These are evaluation formulas for the original metrics, with the same $10^{-20}$ log floor as in Appendix~\ref{app:controller-observations}.
They use only scaled residuals, cached products, and diagonal scalings, requiring no additional multiplication by $A$ or $A^\top$.
The merit $M^{\mathrm{rel}}$, tolerance test, and reward in \Eqref{eq:reward} use these original-coordinate metrics.
Preconditioned and unpreconditioned runs are thus evaluated at the same original-LP tolerance.

\subsection{Matrix--vector product caching}
\label{app:matrix-vector-caching}

For the primal and dual blocks of the vectors used in \Eqref{eq:gen_acc_iter}, we retain
\[
    (Ax_t,\;A\hat x_t,\;Ax_\star,\;A\mathbf K_t^x)
    \qquad\text{and}\qquad
    (A^\top y_t,\;A^\top\hat y_t,\;A^\top y_\star,\;A^\top\mathbf K_t^y).
\]
Applying the acceleration coefficients to these products reconstructs $A\hat x_{t+1}$ and $A^\top\hat y_{t+1}$ without another application of either matrix.
The newest history products are $Ax_{t+1}-Ax_t$ and $A^\top y_{t+1}-A^\top y_t$. The remaining columns are shifted.
At a restart, the products for the new anchor are copied from those of $z_{t+1}$, and the history matrices and their products are reset together.

For $t\geq1$, define the PDHG direction $q_t:=z_t-\hat z_t$, with primal and dual blocks $q_t^x:=x_t-\hat x_t$ and $q_t^y:=y_t-\hat y_t$.
The implementation applies the matrices to the blocks of $q_{t+1}$.
The only new sparse products are $Aq_{t+1}^x$ for the dual update and $A^\top q_{t+1}^y$ for the residual features.
The output products are recovered as $Ax_{t+1}=A\hat x_{t+1}+Aq_{t+1}^x$ and $A^\top y_{t+1}=A^\top\hat y_{t+1}+A^\top q_{t+1}^y$.
The primal update uses the reconstructed $A^\top\hat y_{t+1}$.
For the observation, we also cache the unweighted Gram matrices $(\mathbf K_t^x)^\top\mathbf K_t^x$ and $(\mathbf K_t^y)^\top\mathbf K_t^y$.
Only their newest row and column require fresh inner products when a history direction is inserted. The remaining submatrices are shifted, and both Gram matrices are zeroed at restarts.
GA-PDHG thus uses two new sparse matrix--vector products per iteration, with $O\left(\left(m+n\right)k+k^2\right)$ additional work for acceleration and cached geometry updates.
\textcolor{black}{GALLOP} additionally evaluates the policy at each iteration.

\subsection{Full \textcolor{black}{GALLOP} algorithm}
\label{app:full-algorithm}

Algorithm~\ref{alg:gen-pdhg-rollout} collects one training episode of GA-PDHG under policy $\pi_\theta$ on an LP instance $I$, storing its $T$ transitions in $\tau$.
The critic $V^\phi$ supplies value estimates for training.
Algorithm~\ref{alg:ppo-gen-pdhg-training} trains the \textcolor{black}{GALLOP} policy, with $\mathcal B$ denoting the buffer of valid transitions from all episodes in a PPO iteration.
Feature definitions and policy losses are provided in Appendix~\ref{app:controller-observations} and Appendix~\ref{app:controller-optimization}, respectively.
Initialization takes one unaccelerated PDHG step before the first controlled iteration.
The history stores only output differences inserted after initialization or the latest restart, with unavailable columns set to zero.
In the history update, $\mathbf K_t(:,1\!:\!k-1)$ denotes its first $k-1$ columns and is empty for $k=1$.
The returned $V_{T+1}^\phi$ is the terminal value or truncation bootstrap used by GAE.
The cap $T_{\max}$ counts controlled iterations and excludes the initial unaccelerated PDHG step.
We assume $\|A\|_2>0$ for the step-size formulas and require $\epsilon_{\mathrm{tol}}\geq10^{-20}$ to respect the numerical floor in the log residuals.

\begin{algorithm}[!htbp]
\caption{\textcolor{black}{GALLOP} training rollout}
\label{alg:gen-pdhg-rollout}
\footnotesize
\begin{algorithmic}[1]
\Require Positive integer cap $T_{\max}$, memory $k\geq1$, tolerance $\epsilon_{\mathrm{tol}}$, and controls and features specified in Appendix~\ref{app:controller-policy} and Appendix~\ref{app:controller-observations}
\Function{\textcolor{black}{GALLOP-Rollout}}{$\pi_\theta,V^\phi,I$}
\State Initialize $z_\star=\hat z_1=z_0=0$, weight $\omega_0$, $s\gets1$, $t\gets1$, $\tau\gets\emptyset$
\State $z_1\gets\mathcal H\left(\hat z_1;\omega_0,\eta_0\right)$, $\mathbf K_1\gets0$
\State $M_\star\gets M_1^{\mathrm{fp}}$ by \Eqref{eq:gen_fpr}; initialize observation baselines as in Appendix~\ref{app:controller-observations}
\State Form $o_1$ using Appendix~\ref{app:controller-observations}
\While{$t\leq T_{\max}$}
\State Evaluate $y_t^{\mathrm{prior}}$ and $V^\phi(o_t)$; sample $a_t$ using Appendix~\ref{app:controller-policy}
\State Record $\boldsymbol\ell_t$ by \Eqref{eq:group_densities}; set $\eta_t,\omega_t$ by \Eqref{eq:learned_eta} and \Eqref{eq:learned_omega}
\If{$e_t=0$}
\State Form $\hat z_{t+1}$ by \Eqref{eq:gen_acc_iter}, using the transformed acceleration coefficients
\Else
\State $\hat z_{t+1}\gets z_t$
\EndIf
\State $z_{t+1}\gets\mathcal H(\hat z_{t+1};\omega_t,\eta_t)$; measure $M_{t+1}^{\mathrm{fp}}$ by \Eqref{eq:gen_fpr}
\If{$e_t=1$}
\State $z_\star\gets z_{t+1}$, $M_\star\gets M_{t+1}^{\mathrm{fp}}$, $\mathbf K_{t+1}\gets0$, $s\gets0$
\Else
\State $\mathbf K_{t+1}\gets\left[z_{t+1}-z_t\quad\mathbf K_t(:,1\!:\!k-1)\right]$, $s\gets s+1$
\EndIf
\State Evaluate the KKT metrics and $r_t$ by \Eqref{eq:reward}
\State $\mathrm{success}_t\gets\mathbf1\left[M^{\mathrm{rel}}(z_{t+1})\leq\log\epsilon_{\mathrm{tol}}\right]$
\State Append the transition to $\tau$ as in \Eqref{eq:gen_rollout_record}
\State $t\gets t+1$; form $o_t$ using Appendix~\ref{app:controller-observations}
\If{$\mathrm{success}_{t-1}=1$}
\State \textbf{break}
\EndIf
\EndWhile
\State $T\gets t-1$
\If{$\mathrm{success}_T=1$}
\State $V_{T+1}^\phi\gets0$
\Else
\State $V_{T+1}^\phi\gets V^\phi\left(o_{T+1}\right)$ \Comment{bootstrap at the episode cap}
\EndIf
\State \Return $\tau,V_{T+1}^\phi$
\EndFunction
\end{algorithmic}
\end{algorithm}

\begin{algorithm}[!htbp]
\caption{PPO training of \textcolor{black}{GALLOP}}
\label{alg:ppo-gen-pdhg-training}
\footnotesize
\begin{algorithmic}[1]
\Require LP dataset $\mathcal D=\{I^{(i)}\}_{i=1}^N$, number of PPO iterations $J$, history length $k$,
\Statex \hspace{\algorithmicindent} control and standard-deviation bounds, $\gamma,\kappa,\varepsilon^{\mathrm{clip}},\lambda_{\mathrm{prior}},c_V,\chi,\epsilon_{\mathrm{tol}},T_{\max},\epsilon_{\mathrm{geom}}$
\State Initialize five actors, global log standard deviations, critic $V^\phi$, and their optimizers
\For{$j=1,\ldots,J$}
\State $(\theta_{\mathrm{old}},\phi_{\mathrm{old}})\gets(\theta,\phi)$, $\mathcal B\gets\emptyset$
\For{$i=1,\ldots,N$ in batches}
\State $\tau^{(i)},V_{T^{(i)}+1}^{\phi_{\mathrm{old}}}\gets\textsc{\textcolor{black}{GALLOP-Rollout}}(\pi_{\theta_{\mathrm{old}}},V^{\phi_{\mathrm{old}}},I^{(i)})$
\State Compute $\hat\Psi_t^{(i)}$ by \Eqref{eq:gae}, using $V_{T^{(i)}+1}^{\phi_{\mathrm{old}}}$ for the final bootstrap
\State Append valid rollout records and their advantages to $\mathcal B$
\EndFor
\State Compute $\tilde\Psi_t$ and $\hat R_t$ over the full buffer by \Eqref{eq:gae_targets}
\For{several epochs of minibatches from $\mathcal B$}
\State Update all five actors on $\mathcal L_{\mathrm{actor}}$ in \Eqref{eq:restart_prior}
\State Update the critic on $c_V\mathcal L_V$ in \Eqref{eq:ppo_critic_loss}
\EndFor
\EndFor
\State \Return $\theta,\phi$
\end{algorithmic}
\end{algorithm}
\FloatBarrier

\section{Details on Experiments}
\label{app:experiment-details}

\subsection{Implementation details}
\label{app:experiment-implementation}

\paragraph{Hardware and numerical precision.}
All experiments use an NVIDIA GeForce RTX 4090 GPU with 24 GB of memory and an Intel Core i9-13900K CPU.
The implementation uses JAX 0.10.2 with CUDA 12 and Flax 0.12.8; the external baseline uses MPAX 0.2.4.
The policy networks compute in float32, while the PDHG state, residuals, rewards, and advantage estimates use float64.

\paragraph{Optimization settings.}
Table~\ref{tab:training-hyperparameters} lists the settings of the runs that produced the selected checkpoints.
All networks use Adam with a constant learning rate, four PPO epochs per rollout buffer, and a separate gradient-norm clip of $0.5$ per network.
All families use $k=4$, $\gamma=1$, $\kappa=0.99$, PPO clipping $\varepsilon^{\mathrm{clip}}=0.2$, and no entropy bonus.

\begin{table}[!htbp]
\centering
\small
\setlength{\tabcolsep}{6pt}
\caption{Settings of the training runs used in the six-family comparison.}
\label{tab:training-hyperparameters}
\begin{tabular}{@{}lrrrrr@{}}
\toprule
Family & Learning rate & PPO minibatch & $c_V$ & $\lambda_{\mathrm{prior}}$ & $L_\omega$\\
\midrule
Transport & $5\times10^{-4}$ & 16,384 & 0.5 & 0.02 & 4.6\\
Auction & $5\times10^{-4}$ & 4,096 & 0.5 & 0 & 2.0\\
Min-cost flow & $10^{-4}$ & 16,384 & 0.5 & 0.02 & 4.6\\
Robust inventory & $10^{-4}$ & 16,384 & 0.5 & 0.05 & 4.6\\
$L_\infty$ regression & $5\times10^{-4}$ & 16,384 & 0.2 & 0.01 & 4.6\\
Random LP & $10^{-4}$ & 16,384 & 0.5 & 0.02 & 4.6\\
\bottomrule
\end{tabular}
\end{table}

\paragraph{Exploration standard deviations.}
Each continuous control group has a learned global log standard deviation, shared across observations and updated with its actor.
For $j\in\{\alpha,\beta,\zeta,\eta,\omega\}$, let $\ell_j$ denote this learned scalar. The distribution uses
\begin{equation}
    \sigma_j=\exp\left(\clip\left(\ell_j,\log\sigma_j^{\min},\log1.5\right)\right).
\end{equation}
Table~\ref{tab:exploration-sigmas} gives the initialization and lower bounds from the selected runs.
The two relaxation latents share $\sigma_\alpha$, and all $k$ history coordinates share $\sigma_\beta$.
Inference replaces the Gaussian draws by their means.

\begin{table}[!htbp]
\centering
\footnotesize
\setlength{\tabcolsep}{5pt}
\caption{Gaussian exploration settings. Both tuples use the order $\left(\alpha,\beta,\zeta,\eta,\omega\right)$; the common upper bound is $1.5$.}
\label{tab:exploration-sigmas}
\begin{tabular}{@{}lll@{}}
\toprule
Family & Initial standard deviations & Minimum standard deviations\\
\midrule
Transport & $\left(0.3,0.2,0.4,0.4,0.2\right)$ & $\left(0.2,0.1,0.2,0.2,0.1\right)$\\
Auction & $\left(0.4,0.2,0.4,0.6,0.2\right)$ & $\left(0.001,0.001,0.001,0.12,0.001\right)$\\
Min-cost flow & $\left(0.2,0.3,0.3,0.4,0.2\right)$ & $\left(0.05,0.2,0.05,0.05,0.05\right)$\\
Robust inventory & $\left(0.05,0.05,0.07,0.1,0.07\right)$ & $\left(0.01,0.01,0.01,0.02,0.02\right)$\\
$L_\infty$ regression & $\left(0.4,0.2,0.4,0.5,0.2\right)$ & $\left(0.2,0.1,0.05,0.3,0.1\right)$\\
Random LP & $\left(0.1,0.1,0.1,0.1,0.1\right)$ & $\left(0.02,0.01,0.02,0.02,0.01\right)$\\
\bottomrule
\end{tabular}
\end{table}

\paragraph{Batches across problem sizes.}
LPs with the same dimensions are grouped into a size bucket and solved in parallel within that bucket.
Each PPO iteration collects one episode per training LP from every bucket before updating the policy.
The valid transitions are then combined, their advantages are normalized over the full buffer, and shuffled PPO minibatches are formed using the sizes in Table~\ref{tab:training-hyperparameters}.
Thus rollout batches have fixed LP dimensions, while PPO minibatches can mix observations from different problem sizes because the observation and action dimensions depend only on $k$.
We evaluate the same trained policy separately on each held-out and test bucket, including sizes larger than the training sizes listed in Appendix~\ref{app:protocol}.
\FloatBarrier

\end{document}